\documentclass[letterpaper,11pt]{article}

\usepackage[square,numbers]{natbib}

\usepackage[utf8]{inputenc}   
\usepackage[T1]{fontenc}  

\usepackage{pgf,tikz}
\usetikzlibrary{arrows}
\usepackage{multicol}
\usepackage{amsthm}
\usepackage{amsmath}
\usepackage{amssymb}
\usepackage{amsfonts}
\usepackage{stmaryrd}
\usepackage{mathabx}
\usepackage{mathtools}
\usepackage{latexsym}
\usepackage{color}
\usepackage{graphics,graphicx,graphpap}
\usepackage{multirow}
\usepackage{rotating}
\usepackage[new]{old-arrows}
\usepackage[all]{xy}
\usepackage[letterpaper]{geometry}
\usepackage{subfigure}
\usepackage{faktor}
\usepackage{hyperref}

\usepackage{todonotes}

\theoremstyle{plain}

\usepackage{setspace}
\usepackage{tocbibind}

\DeclareMathAlphabet{\mathpzc}{OT1}{pzc}{m}{it}
\newtheorem{theorem}{Theorem}
\newtheorem{prop}[theorem]{Proposition}

\newtheorem{lem}[theorem]{Lemma}
\newtheorem{conj}[theorem]{Conjecture}
\newtheorem{que}[theorem]{Question}

\newcommand{\hocolim}{\operatornamewithlimits{\mathrm{hocolim}}}
\newcommand{\colim}{\operatornamewithlimits{\mathrm{colim}}}

\title{Homotopy type of the independence complex of some categorical products of graphs}
\author{Omar Antolín Camarena, Andrés Carnero Bravo}

\begin{document}
\maketitle
\begin{abstract}
    It was conjectured by Goyal, Shukla and Singh 
    that the independence complex of the categorical product $K_2\times K_3\times K_n$ has 
    the homotopy type of a wedge of $(n-1)(3n-2)$ spheres of dimension $3$. Here we prove this conjecture by calculating the homotopy 
    type of the independence complex of the graphs $C_{3r}\times K_n$ and $K_2\times K_m\times K_n$. 
    For $C_m \times K_n$ when $m$ is not a multiple of $3$, we calculate the homotopy type for $m = 4, 5$ and show that for other values it has to have 
    the homotopy type of a wedge of spheres of at most $2$ consecutive dimensions and maybe some Moore spaces.
\end{abstract}
\tableofcontents
\section{Introduction}
The independence complex of a graph has been studied extensively and in particular for the family of graphs $K_n\oblong K_m$ there has been  
plenty of work on it without calling it the independence complex of this graph. 
This complex is know as the chessboard complex and it is know that for certain 
pairs $(n,m)$, the integral homology of $I(K_n\oblong K_m)$ has torsion (see\citep{MR2022345}). 
One can ask: what about the independence complex for other graphs products? 
In \citep{homotopygoyal} the authors prove that $I(K_n\times K_m)$ is homotopy equivalent to a wedge of circles and conjectured that
$$I\left(K_2\times K_3\times K_n\right)\simeq\bigvee_{(n-1)(3n-2)}\mathbb{S}^3$$
Here we prove and generalize this by calculating the homotopy type of the independence complex of the following families of graphs:
$C_4\times K_n$, $C_5\times K_n$ and  $C_{3r}\times K_n$ for all $r,n$; and
$K_2\times K_n\times K_m$ for all $n,m$.

For this we show a result which is of interest by itself:
We give the homotopy type of the union of complexes $K_1,\dots,K_n$ such that each $K_i$ is either contractible or homotopy equivalent to a wedge of spheres of dimension not less than 
$r$ and for any non-empty subset $S$ of $\underline{n}$, $\displaystyle\bigcap_{i\in S}K_i=K_S$ is either contractible or
homotopy equivalent to a wedge of spheres of some dimension $r_{|S|}$, where $r_2\leq r-1$ and $r_{i+1}=r_i-1$.
    
\section{Preliminaries}
All graphs are simple, no loops or multiedges. For a graph $G$, $V(G)$ is its vertex set and $E(G)$ 
its edge set. For a vertex $v$, $N_G(v)=\{u\in V(G):\;uv\in E(G)\}$ is its open neighbourhood 
and $N_G[v]=N_G(G)\cup\{v\}$ its closed neighborhood, 
we omit the subindex $G$ if there is no risk of confusion.
$K_n$ is the complete graph with vertex set $\{1,\dots,n\}$ and edge set $\{\{i,j\}:\; i\neq j\}$. 
$C_n$ is the cycle of length $n\geq3$ with vertex set $\{u_1,\dots,u_n\}$ and edge set 
$\{u_1u_2,\dots,u_{n-1}u_n,u_nu_1\}$.
The \textit{categorical product} of the graphs $G$ and $H$ is the graph $G\times H$ with vertex set 
$V(G)\times V(H)$ and edge set 
$$\{\{(u_1,v_1),(u_2,v_2)\}:\;\{u_1,u_2\}\in E(G)\wedge\{v_1,v_2\}\in E(H)\}$$

A simplicial complex $K$ is a family of non-empty subsets of an non-empty finite set $V(K)$, the vertices 
of the complex, such that:
\begin{enumerate}
\item for all $v\in V(K)$, $\{v\}\in K$;
\item if $\emptyset\neq \tau \subseteq \sigma$ and $\sigma\in K$, then $\tau\in K$.
\end{enumerate}
Given a simplicial complex $K$ and a simplex $\sigma$, the \textit{link} of $\sigma$ is the subcomplex 
$lk(\sigma)=\{\tau\in K:\;\tau\cap\sigma=\emptyset\;\wedge\;\tau\cup\sigma\in K\}$ and its \textit{star} is
$st(\sigma)=\{\tau\in K:\tau\cup\sigma\in K\}$. For a vertex we will write $lk(v)$ and $st(v)$ insted of $lk(\{v\})$ or 
$st(\{v\})$.

We will not distinguish between a complex and its geometric realization.

A non empty vertex set $S$ of a graph is independent if any two vertices are not adjacent. 
Give a graph $G$ its 
independence complex is the complex given by its independent sets
$$I(G)=\{\sigma\subseteq V(G):\;\sigma\mbox{ is independent}\}$$

\subsection{Homotopy theory tools}
Now we give the tools we will use throughout the paper for calculating homotopy types of independence complexes.
\begin{lem}\citep{engstrom09}\label{vecindad}
If $N(u)\subseteq N(v)$, then $I(G)\simeq I(G-v)$
\end{lem}

The last lemma can be seen as a particular case of part (a) of the following proposition.
\begin{prop}\label{cofseq}\citep{adamsplit}
There is always a cofibre sequence
\begin{equation*}
\xymatrix{
I(G-N_G[v]) \ar@{^(->}[r] & I(G-v) \ar@{^(->}[r] & I(G) \ar@{->}[r] & \Sigma I(G-N_G[v]) \ar@{->}[r] & \cdots  
}
\end{equation*}
In particular
\begin{itemize}
\item[a)] if $I(G-N_G[v])$ is contractible then the natural inclusion $I(G-v)\hookrightarrow I(G)$ is 
a homotopy equivalence,
\item[b)] if $I(G-N_G[v])\hookrightarrow I(G-v)$ is null-homotopic then there is a splitting
$$I(G)\simeq I(G-v)\vee\Sigma I(G-N_G[v]).$$
\end{itemize}
\end{prop}

For a vertex $v$, we define its star cluster as the subcomplex of $I(G)$ given by
$$SC(v)=\bigcup_{u\in N(v)}st(u)$$

\begin{theorem}\label{barmak}\citep{barmak}
Let $G$ be a a graph and let $v$ be a non-isolated vertex of $G$ which is contained in no triangle. Then
$$I(G)\simeq\Sigma(st_{I(G)}(v)\cap SC(v))$$
\end{theorem}

Given the diagram 
\begin{equation*}
    \xymatrix{
    \mathcal{S}\colon & A \ar@{<-}[r]^{g} & C \ar@{->}[r]^{f} & B
    }
\end{equation*}
its pushout is the space 
$$\colim\left(\mathcal{S}\right)=A\sqcup B/\sim$$
where $g(c)\sim f(c)$ for $c\in C$.
The homotopy pushout of $\mathcal{S}$ is the space
$$\hocolim\left(\mathcal{S}\right)=A\sqcup (A\times I)\sqcup B/\sim$$
where $g(c)\sim (c,0)$ and $f(c)\sim (c,1)$.
If one of $g$ or $f$ is a cofibration, then $\colim(\mathcal{S})\simeq\hocolim(\mathcal{S})$ (see Proposition 3.6.17 
\citep{cubicalhomotopy}). 

\begin{theorem}(See 6.2.8 \citep{introhomo})\label{teohomocompush}
If the following diagram is homotopy commutative ($\alpha\circ f\simeq f'\circ\beta$ and 
$\gamma\circ g\simeq g'\circ\beta$)
\begin{equation*}
    \xymatrix{
    \mathcal{S}\colon & X \ar@{->}[d]^{\alpha} & Z \ar@{->}[l]_{f} \ar@{->}[r]^{g} \ar@{->}[d]^{\beta} & Y \ar@{->}[d]^{\gamma}\\
    \mathcal{S}'\colon & X' & Z' \ar@{->}[l]^{f'} \ar@{->}[r]_{g'}& Y'
    }
\end{equation*}
with  $\alpha,\beta,\gamma$ homotopy equivalences. Then $\hocolim(\mathcal{S})\simeq \hocolim(\mathcal{S}')$.
\end{theorem}

For completness we give the proof of the following well-known fact.

\begin{lem}\label{homocolimpegado}
Let $X,Y,Z$ be spaces with maps $f\colon Z\longrightarrow X$ and $g\colon Z\longrightarrow Y$ such that both maps are null-homotopic. Then
$$\hocolim\left(\mathcal{S}\right)\simeq X\vee Y\vee \Sigma Z$$
where 
\begin{equation*}
    \xymatrix{
    \mathcal{S}\colon & Y \ar@{<-}[r]^{g} & Z \ar@{->}[r]^{f} & X
    }
\end{equation*}
\end{lem}
\begin{proof}
We take the diagrams 
\begin{equation*}
    \xymatrix{
    \mathcal{D}_1\colon & \Sigma Z\vee Y \ar@{<-^)}[r] & \Sigma Z \ar@{^(->}[r] & \Sigma Z\vee X
    }
\end{equation*}
and 
\begin{equation*}
    \xymatrix{
    \mathcal{S}: & Y \ar@{<-}[r]^{g} & Z \ar@{->}[r]^{f} & X
    }
\end{equation*}
where the point of the wedge in $X$ and $Y$ are a point for which the constant map is homotopic to $f$ and $g$ respectively. 
Taking the next diagram where all squares are homotopy pushout squares:
\begin{equation*}
\xymatrix{
Z \ar@{->}[r] \ar@{->}[d] & \ast \ar@{->}[r] \ar@{->}[d]& X \ar@{->}[d] \\
\ast \ar@{->}[r] \ar@{->}[d]& \Sigma Z \ar@{^(->}[r] \ar@{_(->}[d]& \Sigma Z\vee X \ar@{->}[d]\\
Y \ar@{->}[r] & \Sigma Z\vee Y \ar@{->}[r] & \hocolim(\mathcal{D}_1)
}
\end{equation*}
Taking the diagram $\mathcal{D}_3$ given by the compositions 
\begin{equation*}
    \xymatrix{
    Z \ar@{->}[r] & \ast \ar@{->}[r]& X\\
    Z \ar@{->}[r] & \ast \ar@{->}[r]& Y
    }
\end{equation*}
the above $2 \times 2$ square diagram shows that $\hocolim(\mathcal{D}_3)\simeq \hocolim(\mathcal{D}_1)$ (see Proposition 3.7.26 \citep{cubicalhomotopy}) 
and by hypothesis we can construct an homotopy commutative 
diagram between $\mathcal{S}$ and $\mathcal{D}_3$ with the identities. Therefore, by Theorem \ref{teohomocompush} 
$\hocolim(\mathcal{D})\simeq X\vee Y\vee\Sigma Z$.
\end{proof}

\begin{prop}\label{homocolimpegadoesferas}
Let $K_1,\dots,K_n$ be $n\geq2$ CW-complexes such that each $K_i$ is contractible or is homotopy equivalent to a wedge of spheres of dimension not less 
than 
$r$ and for any non-empty subset $S$ of $\{1,\dots,n\}$, $K_S := \displaystyle\bigcap_{i\in S}K_i$ is either contractible or 
homotopy equivalent to a wedge of spheres of dimension $r_{|S|}$, where $r_2\leq r-1$ and $r_{i+1}=r_i-1$. Then
$$\bigcup_{i}^nK_n\simeq\bigvee_{i=1}^nX_i$$
where 
$$X_i=\bigvee_{\{j_1,\dots,j_i\}\in\mathcal{P}(\underline{n})}\Sigma^{i-1}(K_{j_1}\cap\cdots\cap K_{j_i})$$
\end{prop}
\begin{proof}
For $n=2$ this is clear by Lemma \ref{homocolimpegado}. For $n\geq3$ we take  
$$A=\bigcup_{i=1}^{n-1}K_i,\;\;C=\bigcup_{i=1}^{n-1}\left(K_i\cap K_n\right)$$
Then
$$\bigcup_{i}^nK_n=\colim\left(A \longhookleftarrow C \longhookrightarrow K_n\right)\simeq\hocolim\left(A \longhookleftarrow C \longhookrightarrow K_n\right)$$
By inductive hypothesis $C$ is homotopy equivalent to a wedge of spheres of dimension at most $r_2$ or contractible, 
therefore the map $C\longhookrightarrow K_n$ is null-homotopic. By inductive hypothesis $A$ is 
homotopy equivalent to a wedge of spheres of dimesion at least $r_2+1$ or contractible, so the map $C\longhookrightarrow A$ 
is null-homotopic. Applying Lemma \ref{homocolimpegado} we obtain the result.
\end{proof}

The following obvious Lemma can be proved with practically the same proof as the last Proposition.
\begin{lem}\label{lempegcontra}
Let $K_1,\dots,K_n$ be $n\geq2$ CW-complexes such that for any $S$ non-empty subset of $\{1,\dots,n\}$, $\displaystyle\bigcap_{i\in S}K_i=K_S$ is a contractible subcomplex. Then
$$\bigcup_{i}^nK_n\simeq\bigvee_{i=1}^nK_i$$
\end{lem}

\section{Independence Complex of $C_k\times K_n$}

In this section we will be focus on  $I(C_k\times K_n)$, and for this we will need the homotopy type of the independence complex of various 
graphs, such as $P_k\times K_n$ (Theorem \ref{homopkkn}). In the first subsection we will focus in cases with $3\leq k\leq5$ and for all 
$k$ for $n=2$. In the secod section we will proof one of our main results, Theorem \ref{teomod3}, we will star that subsection with a brief sumamry 
of the proof.

The case $C_5\times K_2$ will allows us to find a counterexample showing that the homotopy type 
of the independence complex of categorical product does not depend only on the homotopy type of the independence complexes of the 
factors. To see this, we take $M_q$ as the union of $q$ disjoint edges, from where we get 
that $G=K_2\times K_2\cong M_2$, $G\times G\cong M_8$ and $C_5\times G$ is equal to two disjoint 
copies of $C_5\times K_2$. Now $I(C_5)\cong\mathbb{S}^1\cong I(G)$ and, by Proposition \ref{propCnK2}, 
$I(C_5\times K_2)\simeq\mathbb{S}^2$, therefore
$I(C_5\times G)\simeq\mathbb{S}^5\not\simeq\mathbb{S}^7\cong I(G\times G)$. 

\subsection{$C_3\times K_n,\;\;C_4\times K_n,\;\;C_5\times K_n$ and $C_k\times K_2$}
In \citep{homotopygoyal} the authors use discrete Morse theory to show that the homotopy type of the categorical product of 
complete graphs is the wedge of $\mathbb{S}^1$. For completeness we give here a short proof using simpler tools.
\begin{prop}\label{propcompl}\citep{homotopygoyal}
$$I(K_{n_1}\times K_{n_2})\simeq\bigvee_{(n_1-1)(n_2-1)}\mathbb{S}^1$$
\end{prop}
\begin{proof}
Taking $V(K_{n_1})=\{1,2,\dots,n_1\}$ and $V(K_{n_2})=\{1,2,\dots,n_2\}$, then 
$$V(K_{n_1}\times K_{n_2})=\{(i,j):\;1\leq i\leq n_1\wedge1\leq j\leq n_2\}$$
and, because the definition of the categorical product, the maximal simplices of 
$I(K_{n_1}\times K_{n_2})$ are as:
$$\sigma_i=\{(i,1),(i,2),\dots,(i,n_2)\}\;\;\;\mbox{ or }\;\;\;\tau_j=\{(1,j),(2,j),\dots,(n_1,j)\}$$
Now, for $i\neq k$ and $j\neq l$ we have: $\sigma_i\cap\sigma_k=\emptyset$, $\tau_j\cap\tau_l=\emptyset$, 
$\sigma_i\cap\tau_j=\{(i,j)\}$. By the Nerve Theorem (see \citep{bjornertopmeth} Theorem 10.6), 
we get that 
$$I(K_{n_1}\times K_{n_2})\simeq K_{n_1,n_2}$$
Which is easy to see has the homotopy type of the wedge of $(n_1-1)(n_2-1)$ copies of $\mathbb{S}^1$.
\end{proof}

From this, for $K_3\cong C_3$, we have that
$$I(C_3\times K_n)\simeq\bigvee_{2(n-1)}\mathbb{S}^1$$ 
Proposition \ref{propcompl} can also be used to calculate the homotopy type of $I(C_4\times K_n)$. Taking 
$V(C_4)=\{u_1,u_2,u_3,u_4\}$ and $V(K_n)$ as before, $C_4\times K_n$ is such that 
$N_{C_4\times K_n}((u_1,i))=N_{C_4\times K_n}((u_3,i))$ and 
$N_{C_4\times K_n}((u_2,i))=N_{C_4\times K_n}((u_4,i))$ for all $1\leq i\leq n$. We define 
$$H=C_4\times K_n-(\{u_3\}\times V(K_n))\cup(\{u_4\}\times V(K_n))$$
Now, $H\cong K_2\times K_n$, so 
$$I(C_4\times K_n)\simeq I(H)\simeq\bigvee_{n-1}\mathbb{S}^1$$ 
In fact, if $N_G(u)\subset N_G(v)$ then $I(G\times H)\simeq I(G-v\times V(H))$ for any $H$, therefore 
$I(C_4\times H)\simeq I(K_2\times H)$ for any $H$.

It is easy to see that:
$$C_n\times K_2\cong\left\lbrace 
\begin{array}{cc}
2C_n & \mbox{if } n\equiv0\;(mod\;2)\\
C_{2n} & \mbox{if } n\equiv1\;(mod\;2)
\end{array}\right.
$$
for any $n\geq3$, 
because $C_n\times K_2$ is a $2$-regular bipartite graph, so it is an even cycle or the disjoint union of 
even cycles. By Weichsel’s Theorem (see \citep{handbookprod} Theorem 5.9), $C_n\times K_2$ is connected 
if and only if one of the graph has an odd cycle, and if both graphs are bipartite, the product has 
exactly two connected components 
(see figure \ref{path}.). 
From this, we get the next lemma.

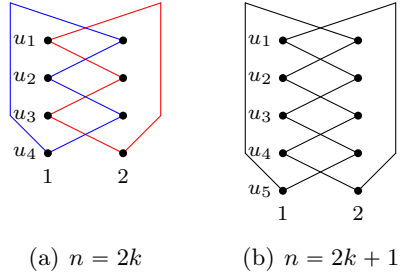
\begin{figure}
\centering
\subfigure[$n=2k$]{\begin{tikzpicture}[line cap=round,line join=round,>=triangle 45,x=1cm,y=1cm]
\clip(-1.5,-1.5) rectangle (1.5,1.5);
\draw[color=red] (-0.5,1)-- (0.5,0.5);
\draw[color=red] (0.5,0.5)-- (-0.5,0);
\draw[color=red] (-0.5,0)-- (0.5,-0.5);
\draw[color=red] (0.5,-0.5)-- (1,0);
\draw[color=red] (1,0)-- (1,1.5);
\draw[color=red] (1,1.5)-- (-0.5,1);
\draw[color=blue] (0.5,1)-- (-0.5,0.5);
\draw[color=blue] (-0.5,0.5)-- (0.5,0);
\draw[color=blue] (0.5,0)-- (-0.5,-0.5);
\draw[color=blue] (-0.5,-0.5)-- (-1,0);
\draw[color=blue] (-1,0)-- (-1,1.5);
\draw[color=blue] (-1,1.5)-- (0.5,1);
\begin{scriptsize}
\fill [color=black] (-0.5,1) circle (1.5pt);
\fill [color=black] (0.5,1) circle (1.5pt);
\fill [color=black] (0.5,0.5) circle (1.5pt);
\fill [color=black] (-0.5,0.5) circle (1.5pt);
\fill [color=black] (0.5,0) circle (1.5pt);
\fill [color=black] (-0.5,0) circle (1.5pt);
\fill [color=black] (0.5,-0.5) circle (1.5pt);
\fill [color=black] (-0.5,-0.5) circle (1.5pt);
\draw[color=black] (-0.8,1) node {$u_1$};
\draw[color=black] (-0.8,0.5) node {$u_2$};
\draw[color=black] (-0.8,0) node {$u_3$};
\draw[color=black] (-0.8,-0.5) node {$u_4$};
\draw[color=black] (-0.5,-0.8) node {$1$};
\draw[color=black] (0.5,-0.8) node {$2$};
\end{scriptsize}
\end{tikzpicture}}
\subfigure[$n=2k+1$]{\begin{tikzpicture}[line cap=round,line join=round,>=triangle 45,x=1cm,y=1cm]
\clip(-1.5,-1.5) rectangle (1.5,1.5);
\draw (-0.5,1)-- (0.5,0.5);
\draw (0.5,0.5)-- (-0.5,0);
\draw (-0.5,0)-- (0.5,-0.5);
\draw (0.5,-0.5)-- (-0.5,-1);
\draw (-0.5,-1)-- (-1,-0.5);
\draw (-1,-0.5)-- (-1,1.5);
\draw (-1,1.5)-- (0.5,1);
\draw (0.5,1)-- (-0.5,0.5);
\draw (-0.5,0.5)-- (0.5,0);
\draw (0.5,0)-- (-0.5,-0.5);
\draw (-0.5,-0.5)-- (0.5,-1);
\draw (0.5,-1)-- (1,-0.5);
\draw (1,-0.5)-- (1,1.5);
\draw (1,1.5)-- (-0.5,1);
\begin{scriptsize}
\fill [color=black] (-0.5,1) circle (1.5pt);
\fill [color=black] (0.5,1) circle (1.5pt);
\fill [color=black] (0.5,0.5) circle (1.5pt);
\fill [color=black] (-0.5,0.5) circle (1.5pt);
\fill [color=black] (0.5,0) circle (1.5pt);
\fill [color=black] (-0.5,0) circle (1.5pt);
\fill [color=black] (0.5,-0.5) circle (1.5pt);
\fill [color=black] (-0.5,-0.5) circle (1.5pt);
\fill [color=black] (0.5,-1) circle (1.5pt);
\fill [color=black] (-0.5,-1) circle (1.5pt);
\draw[color=black] (-0.8,1) node {$u_1$};
\draw[color=black] (-0.8,0.5) node {$u_2$};
\draw[color=black] (-0.8,0) node {$u_3$};
\draw[color=black] (-0.8,-0.5) node {$u_4$};
\draw[color=black] (-0.8,-1) node {$u_5$};
\draw[color=black] (-0.5,-1.3) node {$1$};
\draw[color=black] (0.5,-1.3) node {$2$};
\end{scriptsize}
\end{tikzpicture}}
\caption{$C_n\times K_2$}\label{path}
\end{figure}

\begin{lem}
$$I(C_n\times K_2)\simeq\left\lbrace 
\begin{array}{cc}
I(C_n)*I(C_n) & \mbox{if } n\equiv0\;(mod\;2)\\
I(C_{2n}) & \mbox{if } n\equiv1\;(mod\;2)
\end{array}\right.$$

\end{lem}
Then, for calculating the homotopy type of the independence complex of $C_n\times K_2$ we only need the 
homotopy type of the independence complexes of cycles. 
\begin{theorem}\citep{kozlovdire}
$$I(C_n)\simeq\left\lbrace
\begin{array}{cc}
    \mathbb{S}^{k-1}\vee\mathbb{S}^{k-1} & \mbox{if } n=3k  \\
    \mathbb{S}^{k-1} & \mbox{if } n=3k+1\\
    \mathbb{S}^{k} & \mbox{if } n=3k+2   
\end{array}\right.$$
\end{theorem}

From all this, the next proposition follows:
\begin{prop}\label{propCnK2}
$$I(C_n\times K_2)\simeq\left\lbrace 
\begin{array}{cc}
\displaystyle\bigvee_{4}\mathbb{S}^{4k-1} & \mbox{if } n=6k\\
\mathbb{S}^{4k} & \mbox{if } n=6k+1 \\
\mathbb{S}^{4k+1} & \mbox{if } n=6k+r \;\mbox{with } r\in\{2,4\}\\
\mathbb{S}^{4k+1}\vee\mathbb{S}^{4k+1} & \mbox{if } n=6k+3 \\
\mathbb{S}^{4k+2} & \mbox{if } n=6k+5
\end{array}\right.
$$
\end{prop}

Now we will calculate the homotopy type of $I(C_5\times K_n)$ for all $n\geq2$.

\begin{prop}
For all $n\geq2$
$$I(C_5\times K_n) \simeq\bigvee_{n-1}\mathbb{S}^2$$
\end{prop}
\begin{proof}
We will see that $I(C_5\times K_{n+1})\simeq I(C_5\times K_n)\vee\mathbb{S}^2$. 
Taking $G_0\cong C_5\times K_{n+1}$, we have:
$$N_{G_0}((u_1,n+1))=\bigcup_{i=1}^n\{(u_2,i),(u_5,i)\}$$
and taking $H_1=G_0-N_{G_0}[(u_1,n+1)]$ we have:
$$N_{H_1}((u_3,n+1))=\bigcup_{i=1}^n\{(u_4,i)\} \subseteq \bigcup_{i=1}^n\{(u_4,i),(u_1,i)\} = N_{H_1}((u_5,n+1))$$
and 
$$N_{H_1}((u_4,n+1))=\bigcup_{i=1}^n\{(u_3,i)\} \subseteq \bigcup_{i=1}^n\{(u_3,i), (u_1,i)\} = N_{H_1}((u_2,n+1)),$$
so that $I(H_1)\simeq I(H_1')$, where $H_1'=H_1-(u_2,n+1)-(u_5,n+1)$. Now, in $H_1'$ all the vertices 
of the form $(u_1,i)$ with $1\leq i\leq n$ are isolated, so $I(H_1')$ is contractible. 
Therefore, by Proposition \ref{cofseq}, $I(G_0)\simeq I(G_1)$, 
with $G_1=G_0-(u_1,n+1)$. We define $H_2=G_1-N_{G_1}[(u_2,n+1)]$, noting that 
$$N_{G_1}((u_2,n+1))=\bigcup_{i=1}^n\{(u_1,i),(u_3,i)\}$$
Then $N_{H_2}((u_2,i))\subseteq N_{H_2}((u_4,j))$ for $1\leq i,j\leq n$ and therefore 
$I(H_2)\simeq I(H_2')$ with $H_2'=H_2-(u_4,1)-(u_4,2)-\cdots-(u_4,n)$. In $H_2'$, $(u_5,n+1)$ is an 
isolated vertex, so $I(H_2')$ is contractible and, by Proposition \ref{cofseq}, $I(G_1)\simeq I(G_2)$, with 
$G_2=G_1-(u_2,n+1)$. 

Now, using the part (b) of Proposition \ref{cofseq}, we will see that 
$|I(G_2)|\simeq|I(W_1)|\vee|\Sigma I(W_2)|$, with $W_1=G_2-(u_3,n+1)$ and 
$W_2=G_2-N_{G_2}[(u_3,n+1)]$. In $W_2$, 
$$N_{W_2}((u_3,i))=\{(u_4,n+1)\}\subseteq N_{W_2}((u_5,j))$$ 
for all $1\leq i,j\leq n$, so $I(W_2)\simeq I(W_2-(u_5,1)-\cdots-(u_5,n))$ and 
$$W_2-(u_5,1)-\cdots-(u_5,n)\cong2K_{1,n}$$ Therefore $I(W_2)\simeq\mathbb{S}^1$.
For $W_1$, we first will see that $I(W_1)\simeq I(W_1-(u_4,n+1))$. For this, we take 
$W=W_1-N_{W_1}[(u_4,n+1)]$. In $W$, 
$$N_W((u_4,i))=\{(u_5,n+1)\}\subseteq N_W((u_1,j))$$
for all $1\leq i,j\leq n$, then 
$$I(W)\simeq I(W-(u_1,1)-\cdots-(u_1,n))\cong I(K_n^c\sqcup K_{1,n})\simeq*$$
Then $I(W_1)\simeq I(W_1-(u_4,n+1))$. In $W'=W_1-(u_4,n+1)$, 
$$N_{W'}((u_5,i))\subseteq N_{W'}((u_5,n+1))$$ for all $1\leq i\leq n$. Then 
$$I(W_1)\simeq I(W'-(u_5,n+1))\cong I(C_5\times K_n)\simeq\bigvee_{n-1}\mathbb{S}^2$$
Therefore, the inclusion $I(W_2)\hookrightarrow I(W_1)$ is null-homotopic and 
$$I(C_5\times K_{n+1})\simeq I(G_2)\simeq I(W_1)\vee\Sigma I(W_2)\simeq\bigvee_{n}\mathbb{S}^2$$
\end{proof}

\subsection{$C_k\times K_n$ for $k\geq 6$}
Now we will prove one of the main results of the paper, Theorem \ref{teomod3}, which tells us the homotopy type of the 
independence complex of $C_{3r}\times K_n$; for the other cycles Theorem \ref{teonotmod3} will give us the connectivity and all 
but two of the reduced homology groups. For this
we will need to calculate the homotopy type of the independence complex of various auxiliary graphs. 
The idea is to use the star cluster of 
a vertex and Theorem \ref{barmak} to get an decomposition of the complexes for 
which Proposition \ref{homocolimpegadoesferas} can be used, so the 
suspension of this union will have the same homotopy type as the independence complex of $C_k\times K_n$. The complexes in the union will 
be isomorphic to the independence complex of the graphs $G_{k,n}$ (Figure \ref{gkn}), which are isomorphic to 
$C_k\times K_n-N[u]-N[v]$ where $u$ and $v$ are adjacent vertices and their independence complexes are isomorphic 
to the intersection of their links. For the homotopy type of this family we will also need the homotopy type of the independence complex of 
$P_k\times K_n$ and the graph family $H_{k,n}$ (Figure \ref{hkn}), for which we will need the independence complex of the family $W_{k,n}$ 
(Figure \ref{wkn}). We also will need to see what the intersection of two or more complexes of the decomposition looks like, and for this we will need 
to see what happens with the independence complexes of other two families: $\mathring{H}_{k,n}$ and $\mathring{W}_{k,n}$. The idea for the 
calculation for the auxiliary families will be to use Lemma \ref{vecindad}, Proposition \ref{cofseq} or Theorem \ref{barmak} and Proposition \ref{homocolimpegadoesferas}.

\begin{theorem}\label{homopkkn}
For $n\geq2$, 
$$I(P_k\times K_n)\simeq\left\lbrace
\begin{array}{cc}
\displaystyle\bigvee_{(n-1)^r}\mathbb{S}^{2r-1} & \mbox{if } k=3r\\
* & \mbox{if } k=3r+1\\
\displaystyle\bigvee_{(n-1)^{r+1}}\mathbb{S}^{2r+1} & \mbox{if } k=3r+2
\end{array}
\right.$$
\end{theorem}
\begin{proof}
The proof is by induction on $k$.
For $k=1$, $I(P_1\times K_n)\simeq*$ for any $n$. For $k=2$, 
$I(P_2\times K_n)=I(K_2\times K_n)$ and by Theorem \ref{propcompl} the homotopy type is as claimed. 
For $k=3$, 
$$I(P_k\times K_n)\simeq I(P_k-\{(u_3,i):\;1\leq i\leq n\}\times K_n)\cong I(K_2\times K_n)$$
Supose that for any $r\leq k$ the theorem is true. 
$$I(P_{k+1}\times K_n)\simeq I(P_{k+1}-\{(u_3,i):\;1\leq i\leq n\}\times K_n)\cong I(P_{k-2}\times K_n\sqcup K_2\times K_n)$$
Now
$$I(P_{k-2}\times K_n\sqcup K_2\times K_n)\simeq I(P_{k-2}\times K_n)*\bigvee_{n-1}\mathbb{S}^1\simeq\bigvee_{n-1}\Sigma^2I(P_{k-2}\times K_n)$$
The rest follows by induction.
\end{proof}

For $k\geq2$ and $n\geq3$ we define: 
\begin{itemize}
\item $W_{k,n}$ as the graph obtained from $P_k\times K_n$ by ading 
two new vertices $v_1,v_2$ and the edges $\{\{(u_1,i),v_1\}:\;i\neq2\}\cup\{\{(u_k,i),v_2\}:\;i\neq2\}$ (Figure \ref{wkn}).
\item $H_{k,n}$ as the graph obtained from $P_k\times K_n$ by ading 
two new vertices $v_1,v_2$ and the edges $\{\{(u_1,i),v_1\}:\;i\geq2\}\cup\{\{(u_k,i),v_2\}:\;i\neq2\}$ (Figure \ref{hkn}).
\item $G_{k,n}$ as the graph obtained from $H_{k,n}$ by ading two new vertices $w_1,w_2$ and the edges 
$\{v_1,w_1\},\{w_1,w_2\},\{w_2,v_2\}$ (Figure \ref{gkn}).
\end{itemize}

\begin{figure}
\centering
\subfigure[$W_{k,n}$]{\begin{tikzpicture}[line cap=round,line join=round,>=triangle 45,x=1cm,y=1cm]
\clip(-1.2,-1.5) rectangle (1,2);
\draw (-1,1)-- (-0.5,0.5);
\draw (-0.5,1)-- (0,0.5);
\draw (-0.5,1)-- (-1,0.5);
\draw (0,1)-- (-0.5,0.5);
\draw (-1,-0.5)-- (-0.5,0);
\draw (-0.5,-0.5)-- (0,0);
\draw (-0.5,-0.5)-- (-1,0);
\draw (0,-0.5)-- (-0.5,0);
\draw (-0.5,1.5)-- (-1,1);
\draw (-0.5,1.5)-- (0,1);
\draw (-0.5,1.5)-- (0.5,1);
\draw (-0.5,-1)-- (-1,-0.5);
\draw (-0.5,-1)-- (0,-0.5);
\draw (-0.5,-1)-- (0.5,-0.5);
\begin{scriptsize}
\fill [color=black] (-1,1) circle (1.5pt);
\fill [color=black] (-0.5,1) circle (1.5pt);
\fill [color=black] (0,1) circle (1.5pt);
\fill [color=black] (0.5,1) circle (1.5pt);
\fill [color=black] (-1,0.5) circle (1.5pt);
\fill [color=black] (-0.5,0.5) circle (1.5pt);
\fill [color=black] (0,0.5) circle (1.5pt);
\fill [color=black] (0.5,0.5) circle (1.5pt);
\fill [color=black] (-1,-0.5) circle (1.5pt);
\fill [color=black] (-0.5,-0.5) circle (1.5pt);
\fill [color=black] (0,-0.5) circle (1.5pt);
\fill [color=black] (0.5,-0.5) circle (1.5pt);
\fill [color=black] (-1,0) circle (1.5pt);
\fill [color=black] (-0.5,0) circle (1.5pt);
\fill [color=black] (0,0) circle (1.5pt);
\fill [color=black] (0.5,0) circle (1.5pt);
\fill [color=black] (-0.5,1.5) circle (1.5pt);
\fill [color=black] (-0.5,1) circle (1.5pt);
\fill [color=black] (-0.5,-1) circle (1.5pt);
\draw[color=black] (-0.5,1.7) node {$v_1$};
\draw[color=black] (-0.5,-1.2) node {$v_2$};
\draw[color=black] (0.25,1) node {$\cdots$};
\draw[color=black] (0.25,0.5) node {$\cdots$};
\draw[color=black] (0.25,-0.5) node {$\cdots$};
\draw[color=black] (0.25,0) node {$\cdots$};
\draw[color=black] (-1,0.25) node {$\vdots$};
\draw[color=black] (-0.5,0.25) node {$\vdots$};
\draw[color=black] (0,0.25) node {$\vdots$};
\draw[color=black] (0.5,0.25) node {$\vdots$};
\end{scriptsize}
\end{tikzpicture}\label{wkn}}
\subfigure[$H_{k,n}$]{\begin{tikzpicture}[line cap=round,line join=round,>=triangle 45,x=1cm,y=1cm]
\clip(-1.2,-1.5) rectangle (1,2);
\draw (-1,1)-- (-0.5,0.5);
\draw (-0.5,1)-- (0,0.5);
\draw (-0.5,1)-- (-1,0.5);
\draw (0,1)-- (-0.5,0.5);
\draw (-1,-0.5)-- (-0.5,0);
\draw (-0.5,-0.5)-- (0,0);
\draw (-0.5,-0.5)-- (-1,0);
\draw (0,-0.5)-- (-0.5,0);
\draw (-1,1.5)-- (-0.5,1);
\draw (-1,1.5)-- (0,1);
\draw (-1,1.5)-- (0.5,1);
\draw (-0.5,-1)-- (-1,-0.5);
\draw (-0.5,-1)-- (0,-0.5);
\draw (-0.5,-1)-- (0.5,-0.5);
\begin{scriptsize}
\fill [color=black] (-1,1) circle (1.5pt);
\fill [color=black] (-0.5,1) circle (1.5pt);
\fill [color=black] (0,1) circle (1.5pt);
\fill [color=black] (0.5,1) circle (1.5pt);
\fill [color=black] (-1,0.5) circle (1.5pt);
\fill [color=black] (-0.5,0.5) circle (1.5pt);
\fill [color=black] (0,0.5) circle (1.5pt);
\fill [color=black] (0.5,0.5) circle (1.5pt);
\fill [color=black] (-1,-0.5) circle (1.5pt);
\fill [color=black] (-0.5,-0.5) circle (1.5pt);
\fill [color=black] (0,-0.5) circle (1.5pt);
\fill [color=black] (0.5,-0.5) circle (1.5pt);
\fill [color=black] (-1,0) circle (1.5pt);
\fill [color=black] (-0.5,0) circle (1.5pt);
\fill [color=black] (0,0) circle (1.5pt);
\fill [color=black] (0.5,0) circle (1.5pt);
\fill [color=black] (-1,1.5) circle (1.5pt);
\fill [color=black] (-0.5,1) circle (1.5pt);
\fill [color=black] (-0.5,-1) circle (1.5pt);
\draw[color=black] (-1,1.7) node {$v_1$};
\draw[color=black] (-0.5,-1.2) node {$v_2$};
\draw[color=black] (0.25,1) node {$\cdots$};
\draw[color=black] (0.25,0.5) node {$\cdots$};
\draw[color=black] (0.25,-0.5) node {$\cdots$};
\draw[color=black] (0.25,0) node {$\cdots$};
\draw[color=black] (-1,0.25) node {$\vdots$};
\draw[color=black] (-0.5,0.25) node {$\vdots$};
\draw[color=black] (0,0.25) node {$\vdots$};
\draw[color=black] (0.5,0.25) node {$\vdots$};
\end{scriptsize}
\end{tikzpicture}\label{hkn}}
\subfigure[$G_{k,n}$]{\begin{tikzpicture}[line cap=round,line join=round,>=triangle 45,x=1cm,y=1cm]
\clip(-1.2,-1.5) rectangle (1.7,2);
\draw (-1,1)-- (-0.5,0.5);
\draw (-0.5,1)-- (0,0.5);
\draw (-0.5,1)-- (-1,0.5);
\draw (0,1)-- (-0.5,0.5);
\draw (-1,-0.5)-- (-0.5,0);
\draw (-0.5,-0.5)-- (0,0);
\draw (-0.5,-0.5)-- (-1,0);
\draw (0,-0.5)-- (-0.5,0);
\draw (-1,1.5)-- (-0.5,1);
\draw (-1,1.5)-- (0,1);
\draw (-1,1.5)-- (0.5,1);
\draw (-1,1.5)-- (1,1.5);
\draw (-0.5,-1)-- (-1,-0.5);
\draw (-0.5,-1)-- (0,-0.5);
\draw (-0.5,-1)-- (0.5,-0.5);
\draw (-0.5,-1)-- (1,-1);
\draw (1,-1)-- (1,1.5);
\begin{scriptsize}
\fill [color=black] (-1,1) circle (1.5pt);
\fill [color=black] (-0.5,1) circle (1.5pt);
\fill [color=black] (0,1) circle (1.5pt);
\fill [color=black] (0.5,1) circle (1.5pt);
\fill [color=black] (-1,0.5) circle (1.5pt);
\fill [color=black] (-0.5,0.5) circle (1.5pt);
\fill [color=black] (0,0.5) circle (1.5pt);
\fill [color=black] (0.5,0.5) circle (1.5pt);
\fill [color=black] (-1,-0.5) circle (1.5pt);
\fill [color=black] (-0.5,-0.5) circle (1.5pt);
\fill [color=black] (0,-0.5) circle (1.5pt);
\fill [color=black] (0.5,-0.5) circle (1.5pt);
\fill [color=black] (-1,0) circle (1.5pt);
\fill [color=black] (-0.5,0) circle (1.5pt);
\fill [color=black] (0,0) circle (1.5pt);
\fill [color=black] (0.5,0) circle (1.5pt);
\fill [color=black] (-1,1.5) circle (1.5pt);
\fill [color=black] (-0.5,1) circle (1.5pt);
\fill [color=black] (1,1.5) circle (1.5pt);
\fill [color=black] (1,-1) circle (1.5pt);
\fill [color=black] (-0.5,-1) circle (1.5pt);
\draw[color=black] (-1,1.7) node {$v_1$};
\draw[color=black] (-0.5,-1.2) node {$v_2$};
\draw[color=black] (1,1.7) node {$w_1$};
\draw[color=black] (1,-1.2) node {$w_2$};
\draw[color=black] (0.25,1) node {$\cdots$};
\draw[color=black] (0.25,0.5) node {$\cdots$};
\draw[color=black] (0.25,-0.5) node {$\cdots$};
\draw[color=black] (0.25,0) node {$\cdots$};
\draw[color=black] (-1,0.25) node {$\vdots$};
\draw[color=black] (-0.5,0.25) node {$\vdots$};
\draw[color=black] (0,0.25) node {$\vdots$};
\draw[color=black] (0.5,0.25) node {$\vdots$};
\end{scriptsize}
\end{tikzpicture}\label{gkn}}
\caption{}
\end{figure}

\begin{lem}
$$I(W_{k,n})\simeq\left\lbrace
\begin{array}{cc}
\displaystyle\bigvee_{n-1}\mathbb{S}^1 & \mbox{if } k=2\\
\Sigma I(W_{k-1,n})& \mbox{if } k=3r \mbox{ and } r\geq1\\
\Sigma^2I(W_{k-2,n})& \mbox{if } k=3r+1 \mbox{ and } r\geq1\\
\displaystyle\bigvee_{n-1}\Sigma^2I(H_{k-3,n}) & \mbox{if } k=3r+2 \mbox{ and } r\geq1
\end{array}
\right.$$
\end{lem}
\begin{proof}
For $k=2$, $N(v_1)=N((u_2,2))$ and $N(v_2)=N((u_1,2))$. Therefore 
$$I(W_{2,n})\simeq I(W_{2,n}-v_1-v_2)\cong I(K_2\times K_n)\simeq\bigvee_{n-1}\mathbb{S}^1$$

For $k=3r$, in $W_{3r,n}-v_1$ we have that $N((u_1,i))\subseteq N((u_3,i))$ for all $1\leq i\leq n$, 
therefore 
$$I(W_{3r,n}-v_1)\simeq I(W_1)$$
with $W_1=W_{3r,n}-\left\lbrace(u_3,i)):\; 1\leq i\leq n\right\rbrace$. In $W_1$, we have that 
$N((u_4,i))\subseteq N((u_6,i))$ for all $1\leq i\leq n$, 
therefore 
$$I(W_1)\simeq I(W_2)$$
with $W_2=W_1-\left\lbrace(u_6,i)):\; 1\leq i\leq n\right\rbrace$. We keep doing this until we have erased 
all the vertices of the form $(u_{3j},i)$ for $1\leq j\leq r$ and $1\leq i \leq n$, in this new graph $W_{3r}$ 
the vertex $v_2$ is isolated, and thus $I(W_{3r})\simeq*$. Therefore
$$I(W_{3r,n})\simeq\Sigma I\left(W_{3r,n}-N[v_1]\right)\cong\Sigma I(W_{3r-1,n})$$

For $k=3r+1$, we do the same as in the last case, we take $W_{3r+1,n}-v_1$ and erase all the vertices 
of the form $(u_{3j},i)$ for $1\leq j\leq r$ and $1\leq i \leq n$, and call this graph $W_{3r}$. In $W_{3r}$, 
the vertex $(u_{3r+1},2)$ is an isolated vertex, therefore $I(W_{3r})\simeq*$ and 
$$I(W_{3r+1,n})\simeq\Sigma I\left(W_{3r+1,n}-N[v_1]\right)\cong\Sigma I(W_{3r,n})\simeq\Sigma^2I(W_{3r-2,n})$$

\begin{figure}
\centering
\subfigure[$G_{v_1}$]{\begin{tikzpicture}[line cap=round,line join=round,>=triangle 45,x=1cm,y=1cm]
\clip(-1.2,-2) rectangle (1,2);
\draw (-0.5,1)-- (0,0.5);
\draw (-0.5,1)-- (-1,0.5);
\draw (-1,-0.5)-- (-0.5,0);
\draw (-0.5,-0.5)-- (0,0);
\draw (-0.5,-0.5)-- (-1,0);
\draw (0,-0.5)-- (-0.5,0);
\draw (-0.5,-1.5)-- (-1,-1);
\draw (-0.5,-1.5)-- (0,-1);
\draw (-0.5,-1.5)-- (0.5,-1);
\draw (-1,-0.5)-- (-0.5,-1);
\draw (-0.5,-0.5)-- (-1,-1);
\draw (-0.5,-0.5)-- (0,-1);
\draw (-0.5,-1)-- (0,-0.5);
\draw (-1,0.5)-- (-0.5,0);
\draw (-1,0)-- (-0.5,0.5);
\draw (-0.5,0.5)-- (0,0);
\draw (-0.5,0)-- (0,0.5);
\begin{scriptsize}
\fill [color=black] (-0.5,1) circle (1.5pt);
\fill [color=black] (-1,0.5) circle (1.5pt);
\fill [color=black] (-0.5,0.5) circle (1.5pt);
\fill [color=black] (0,0.5) circle (1.5pt);
\fill [color=black] (0.5,0.5) circle (1.5pt);
\fill [color=black] (-1,-0.5) circle (1.5pt);
\fill [color=black] (-0.5,-0.5) circle (1.5pt);
\fill [color=black] (0,-0.5) circle (1.5pt);
\fill [color=black] (0.5,-0.5) circle (1.5pt);
\fill [color=black] (-1,0) circle (1.5pt);
\fill [color=black] (-0.5,0) circle (1.5pt);
\fill [color=black] (0,0) circle (1.5pt);
\fill [color=black] (0.5,0) circle (1.5pt);
\fill [color=black] (-0.5,1.5) circle (1.5pt);
\fill [color=black] (-0.5,1) circle (1.5pt);
\fill [color=black] (-0.5,-1.5) circle (1.5pt);
\fill [color=black] (-1,-1) circle (1.5pt);
\fill [color=black] (-0.5,-1) circle (1.5pt);
\fill [color=black] (0,-1) circle (1.5pt);
\fill [color=black] (0.5,-1) circle (1.5pt);
\draw[color=black] (-0.5,1.7) node {$v_1$};
\draw[color=black] (-0.5,-1.7) node {$v_2$};
\draw[color=black] (0.25,0.5) node {$\cdots$};
\draw[color=black] (0.25,-0.5) node {$\cdots$};
\draw[color=black] (0.25,0) node {$\cdots$};
\end{scriptsize}
\end{tikzpicture}\label{stwknv1}}
\subfigure[$G_{(u_1,1)}$]{\begin{tikzpicture}[line cap=round,line join=round,>=triangle 45,x=1cm,y=1cm]
\clip(-1.2,-2) rectangle (1,2);
\draw (-0.5,1)-- (-1,0.5);
\draw (0,1)-- (-1,0.5);
\draw (-1,0.5)-- (0.5,1);
\draw (-1,-0.5)-- (-0.5,0);
\draw (-0.5,-0.5)-- (0,0);
\draw (-0.5,-0.5)-- (-1,0);
\draw (0,-0.5)-- (-0.5,0);
\draw (-0.5,-1.5)-- (-1,-1);
\draw (-0.5,-1.5)-- (0,-1);
\draw (-0.5,-1.5)-- (0.5,-1);
\draw (-1,-0.5)-- (-0.5,-1);
\draw (-0.5,-0.5)-- (-1,-1);
\draw (-0.5,-0.5)-- (0,-1);
\draw (-0.5,-1)-- (0,-0.5);
\draw (-1,0.5)-- (-0.5,0);
\draw (-1,0.5)-- (0,0);
\draw (-1,0.5)-- (0.5,0);
\begin{scriptsize}
\fill [color=black] (-1,1) circle (1.5pt);
\fill [color=black] (-0.5,1) circle (1.5pt);
\fill [color=black] (0,1) circle (1.5pt);
\fill [color=black] (0.5,1) circle (1.5pt);
\fill [color=black] (-1,0.5) circle (1.5pt);
\fill [color=black] (-1,-0.5) circle (1.5pt);
\fill [color=black] (-0.5,-0.5) circle (1.5pt);
\fill [color=black] (0,-0.5) circle (1.5pt);
\fill [color=black] (0.5,-0.5) circle (1.5pt);
\fill [color=black] (-1,0) circle (1.5pt);
\fill [color=black] (-0.5,0) circle (1.5pt);
\fill [color=black] (0,0) circle (1.5pt);
\fill [color=black] (0.5,0) circle (1.5pt);
\fill [color=black] (-0.5,1) circle (1.5pt);
\fill [color=black] (-0.5,-1.5) circle (1.5pt);
\fill [color=black] (-1,-1) circle (1.5pt);
\fill [color=black] (-0.5,-1) circle (1.5pt);
\fill [color=black] (0,-1) circle (1.5pt);
\fill [color=black] (0.5,-1) circle (1.5pt);
\draw[color=black] (-0.5,-1.7) node {$v_2$};
\draw[color=black] (0.25,1) node {$\cdots$};
\draw[color=black] (0.25,-0.5) node {$\cdots$};
\draw[color=black] (0.25,0) node {$\cdots$};
\end{scriptsize}
\end{tikzpicture}\label{stwknw}}
\subfigure[$T$]{\begin{tikzpicture}[line cap=round,line join=round,>=triangle 45,x=1cm,y=1cm]
\clip(-1.2,-2) rectangle (1,2);
\draw (-0.5,1)-- (-1,0.5);
\draw (-1,-0.5)-- (-0.5,0);
\draw (-0.5,-0.5)-- (0,0);
\draw (-0.5,-0.5)-- (-1,0);
\draw (0,-0.5)-- (-0.5,0);
\draw (-0.5,-1.5)-- (-1,-1);
\draw (-0.5,-1.5)-- (0,-1);
\draw (-0.5,-1.5)-- (0.5,-1);
\draw (-1,-0.5)-- (-0.5,-1);
\draw (-0.5,-0.5)-- (-1,-1);
\draw (-0.5,-0.5)-- (0,-1);
\draw (-0.5,-1)-- (0,-0.5);
\draw (-1,0.5)-- (-0.5,0);
\draw (-1,0.5)-- (0,0);
\draw (-1,0.5)-- (0.5,0);
\begin{scriptsize}
\fill [color=black] (-0.5,1) circle (1.5pt);
\fill [color=black] (-1,0.5) circle (1.5pt);
\fill [color=black] (-1,-0.5) circle (1.5pt);
\fill [color=black] (-0.5,-0.5) circle (1.5pt);
\fill [color=black] (0,-0.5) circle (1.5pt);
\fill [color=black] (0.5,-0.5) circle (1.5pt);
\fill [color=black] (-1,0) circle (1.5pt);
\fill [color=black] (-0.5,0) circle (1.5pt);
\fill [color=black] (0,0) circle (1.5pt);
\fill [color=black] (0.5,0) circle (1.5pt);
\fill [color=black] (-0.5,1) circle (1.5pt);
\fill [color=black] (-0.5,-1) circle (1.5pt);
\fill [color=black] (-0.5,-1.5) circle (1.5pt);
\fill [color=black] (-1,-1) circle (1.5pt);
\fill [color=black] (-0.5,-1) circle (1.5pt);
\fill [color=black] (0,-1) circle (1.5pt);
\fill [color=black] (0.5,-1) circle (1.5pt);
\draw[color=black] (-0.5,-1.7) node {$v_2$};
\draw[color=black] (0.25,-0.5) node {$\cdots$};
\draw[color=black] (0.25,0) node {$\cdots$};
\end{scriptsize}
\end{tikzpicture}\label{stwknv1capstwknw}}
\caption{}
\end{figure}

For $k=3r+2$, by Theorem \ref{barmak}, we have that 
$$I(W_{k,n})\simeq\Sigma\left(st(v_1)\cap SC(v_1)\right)$$ 
Now 
$$st(v_1)\cap SC(v_1)=\bigcup_{w\in N_{W_{k,n}}(v_1)}\left(st(v_1)\cap st(w)\right)$$

For any vertex $w$, $st(w)\cong I(G_w)$, with  
$G_w=W_{k,n}-N(w)$ (Figures \ref{stwknv1},\ref{stwknw}).

For any neighbor of $v_1$, $st(v_1)\cap st(w)\cong st(v_1)\cap st((u_1,1))=I(T)$ 
with 
$$T=W_{k,n}-\left(N_{W_{k,n}}[v_1]\cup N_{W_{k,n}}((u_1,1))\right)$$
(Figure \ref{stwknv1capstwknw}). Now, because 
$N_{T}((u_1,2))\subset N_{T}((u_i,3))$ for any $i\geq2$, we see that 
$$I(T)\simeq\Sigma I(H_{k-3,n}).$$
Now, for any $(u_1,i),(u_1,j)$ such that $i$, $j$ and $2$ are three distinct numbers, if we set $K_i=st(v_1)\cap st((u_1,i))$, 
then $K_i\cap K_j\simeq*$ because it is a cone, the vertex $(u_1,2)$ is an isolated vertex in the corresponding 
subgraph. By Lemma \ref{lempegcontra}, 
$$st_{I(W_{k,n})}(v_1)\cap SC(v_1)\simeq\bigvee_{n-1}\Sigma I(H_{k-3,n})$$
\end{proof}

\begin{lem}\label{lemhkn}
For $k\geq2$, $n\geq3$ and $r\geq2$:
$$I(H_{k,n})\simeq\left\lbrace
\begin{array}{cc}
\displaystyle\Sigma I\left(H_{k-1,n}\right) & \mbox{if } k=3r\\
\displaystyle\Sigma^2I\left(H_{k-2,n}\right) & \mbox{if } k=3r+1\\
\displaystyle\left(\bigvee_{n-1}\Sigma^4I(H_{k-6,n})\right)\vee\left(\bigvee_{n-2}\Sigma^2I(H_{k-3,n})\right) & \mbox{if } k=3r+2
\end{array}
\right.$$
\end{lem}
\begin{proof}
\begin{figure}
\centering
\subfigure[$G$]{\begin{tikzpicture}[line cap=round,line join=round,>=triangle 45,x=1cm,y=1cm]
\clip(-1.2,-2) rectangle (1,2);
\draw (-1,1)-- (-0.5,0.5);
\draw (-0.5,1)-- (0,0.5);
\draw (-0.5,1)-- (-1,0.5);
\draw (0,1)-- (-0.5,0.5);
\draw (-1,-0.5)-- (-0.5,0);
\draw (-0.5,-0.5)-- (0,0);
\draw (-0.5,-0.5)-- (-1,0);
\draw (0,-0.5)-- (-0.5,0);
\draw (-1,1.5)-- (-0.5,1);
\draw (-0.5,1.5)-- (0,1);
\draw (-1,1)-- (-0.5,1.5);
\draw (-0.5,1)-- (0,1.5);
\draw (-1,-1)-- (-0.5,-0.5);
\draw (-0.5,-1)-- (0,-0.5);
\draw (-1,-0.5)-- (-0.5,-1);
\draw (-0.5,-0.5)-- (0,-1);
\draw (-0.5,-1.5)-- (-1,-1);
\draw (-0.5,-1.5)-- (0,-1);
\draw (-0.5,-1.5)-- (0.5,-1);
\begin{scriptsize}
\fill [color=black] (-0.5,1.5) circle (1.5pt);
\fill [color=black] (0,1.5) circle (1.5pt);
\fill [color=black] (0.5,1.5) circle (1.5pt);
\fill [color=black] (-1,1) circle (1.5pt);
\fill [color=black] (-0.5,1) circle (1.5pt);
\fill [color=black] (0,1) circle (1.5pt);
\fill [color=black] (0.5,1) circle (1.5pt);
\fill [color=black] (-1,0.5) circle (1.5pt);
\fill [color=black] (-0.5,0.5) circle (1.5pt);
\fill [color=black] (0,0.5) circle (1.5pt);
\fill [color=black] (0.5,0.5) circle (1.5pt);
\fill [color=black] (-1,-0.5) circle (1.5pt);
\fill [color=black] (-0.5,-0.5) circle (1.5pt);
\fill [color=black] (0,-0.5) circle (1.5pt);
\fill [color=black] (0.5,-0.5) circle (1.5pt);
\fill [color=black] (-1,0) circle (1.5pt);
\fill [color=black] (-0.5,0) circle (1.5pt);
\fill [color=black] (0,0) circle (1.5pt);
\fill [color=black] (0.5,0) circle (1.5pt);
\fill [color=black] (-1,1.5) circle (1.5pt);
\fill [color=black] (-0.5,1) circle (1.5pt);
\fill [color=black] (-0.5,-1.5) circle (1.5pt);
\fill [color=black] (-1,-1) circle (1.5pt);
\fill [color=black] (-0.5,-1) circle (1.5pt);
\fill [color=black] (0,-1) circle (1.5pt);
\fill [color=black] (0.5,-1) circle (1.5pt);
\draw[color=black] (-0.5,-1.8) node {$v_2$};
\draw[color=black] (0.25,1) node {$\cdots$};
\draw[color=black] (0.25,0.5) node {$\cdots$};
\draw[color=black] (0.25,-0.5) node {$\cdots$};
\draw[color=black] (0.25,0) node {$\cdots$};
\draw[color=black] (0.25,1.5) node {$\cdots$};
\draw[color=black] (0.25,-1) node {$\cdots$};
\draw[color=black] (-1,0.25) node {$\vdots$};
\draw[color=black] (-0.5,0.25) node {$\vdots$};
\draw[color=black] (0,0.25) node {$\vdots$};
\draw[color=black] (0.5,0.25) node {$\vdots$};
\end{scriptsize}
\end{tikzpicture}\label{h3rn-v1}}
\subfigure[$H_{3r+2,n}$]{\begin{tikzpicture}[line cap=round,line join=round,>=triangle 45,x=1cm,y=1cm]
\clip(-1.2,-2) rectangle (1,2);
\draw (-1,1)-- (-0.5,0.5);
\draw (-0.5,1)-- (0,0.5);
\draw (-0.5,1)-- (-1,0.5);
\draw (0,1)-- (-0.5,0.5);
\draw (-1,-0.5)-- (-0.5,0);
\draw (-0.5,-0.5)-- (0,0);
\draw (-0.5,-0.5)-- (-1,0);
\draw (0,-0.5)-- (-0.5,0);
\draw (-1,1.5)-- (-0.5,1);
\draw (-1,-1)-- (-0.5,-0.5);
\draw (-0.5,-1)-- (0,-0.5);
\draw (-1,-0.5)-- (-0.5,-1);
\draw (-0.5,-0.5)-- (0,-1);
\draw (-0.5,-1.5)-- (-1,-1);
\draw (-0.5,-1.5)-- (0,-1);
\draw (-0.5,-1.5)-- (0.5,-1);
\draw (-1,1.5)-- (0,1);
\draw (-1,1.5)-- (0.5,1);
\draw (-1,0.5)-- (-0.5,0);
\draw (-0.5,0.5)-- (0,0);
\draw (-1,0)-- (-0.5,0.5);
\draw (-0.5,0)-- (0,0.5);
\begin{scriptsize}
\fill [color=black] (-1,1) circle (1.5pt);
\fill [color=black] (-0.5,1) circle (1.5pt);
\fill [color=black] (0,1) circle (1.5pt);
\fill [color=black] (0.5,1) circle (1.5pt);
\fill [color=black] (-1,0.5) circle (1.5pt);
\fill [color=black] (-0.5,0.5) circle (1.5pt);
\fill [color=black] (0,0.5) circle (1.5pt);
\fill [color=black] (0.5,0.5) circle (1.5pt);
\fill [color=black] (-1,-0.5) circle (1.5pt);
\fill [color=black] (-0.5,-0.5) circle (1.5pt);
\fill [color=black] (0,-0.5) circle (1.5pt);
\fill [color=black] (0.5,-0.5) circle (1.5pt);
\fill [color=black] (-1,0) circle (1.5pt);
\fill [color=black] (-0.5,0) circle (1.5pt);
\fill [color=black] (0,0) circle (1.5pt);
\fill [color=black] (0.5,0) circle (1.5pt);
\fill [color=black] (-1,1.5) circle (1.5pt);
\fill [color=black] (-0.5,1) circle (1.5pt);
\fill [color=black] (-0.5,-1.5) circle (1.5pt);
\fill [color=black] (-1,-1) circle (1.5pt);
\fill [color=black] (-0.5,-1) circle (1.5pt);
\fill [color=black] (0,-1) circle (1.5pt);
\fill [color=black] (0.5,-1) circle (1.5pt);
\draw[color=black] (-1,1.7) node {$v_1$};
\draw[color=black] (-0.5,-1.8) node {$v_2$};
\draw[color=black] (0.25,1) node {$\cdots$};
\draw[color=black] (0.25,0.5) node {$\cdots$};
\draw[color=black] (0.25,-0.5) node {$\cdots$};
\draw[color=black] (0.25,0) node {$\cdots$};
\draw[color=black] (0.25,-1) node {$\cdots$};
\end{scriptsize}
\end{tikzpicture}}
\subfigure[$J_1$]{\begin{tikzpicture}[line cap=round,line join=round,>=triangle 45,x=1cm,y=1cm]
\clip(-1.2,-2) rectangle (1,1.1);
\draw (-1,1)-- (-0.5,0.5);
\draw (-1,-0.5)-- (-0.5,0);
\draw (-0.5,-0.5)-- (0,0);
\draw (-0.5,-0.5)-- (-1,0);
\draw (0,-0.5)-- (-0.5,0);
\draw (-1,-1)-- (-0.5,-0.5);
\draw (-0.5,-1)-- (0,-0.5);
\draw (-1,-0.5)-- (-0.5,-1);
\draw (-0.5,-0.5)-- (0,-1);
\draw (-0.5,-1.5)-- (-1,-1);
\draw (-0.5,-1.5)-- (0,-1);
\draw (-0.5,-1.5)-- (0.5,-1);
\draw (-0.5,0.5)-- (0,0);
\draw (-1,0)-- (-0.5,0.5);
\draw (-0.5,0.5)-- (0.5,0);
\begin{scriptsize}
\fill [color=black] (-1,1) circle (1.5pt);
\fill [color=black] (-0.5,0.5) circle (1.5pt);
\fill [color=black] (-1,-0.5) circle (1.5pt);
\fill [color=black] (-0.5,-0.5) circle (1.5pt);
\fill [color=black] (0,-0.5) circle (1.5pt);
\fill [color=black] (0.5,-0.5) circle (1.5pt);
\fill [color=black] (-1,0) circle (1.5pt);
\fill [color=black] (-0.5,0) circle (1.5pt);
\fill [color=black] (0,0) circle (1.5pt);
\fill [color=black] (0.5,0) circle (1.5pt);
\fill [color=black] (-0.5,-1.5) circle (1.5pt);
\fill [color=black] (-1,-1) circle (1.5pt);
\fill [color=black] (-0.5,-1) circle (1.5pt);
\fill [color=black] (0,-1) circle (1.5pt);
\fill [color=black] (0.5,-1) circle (1.5pt);
\draw[color=black] (-0.5,-1.8) node {$v_2$};
\draw[color=black] (0.25,-0.5) node {$\cdots$};
\draw[color=black] (0.25,0) node {$\cdots$};
\draw[color=black] (0.25,-1) node {$\cdots$};
\end{scriptsize}
\end{tikzpicture}\label{J1}}
\subfigure[$J_2$]{\begin{tikzpicture}[line cap=round,line join=round,>=triangle 45,x=1cm,y=1cm]
\clip(-1.2,-2) rectangle (1,1.1);
\draw (-1,1)-- (0,0.5);
\draw (-1,-0.5)-- (-0.5,0);
\draw (-0.5,-0.5)-- (0,0);
\draw (-0.5,-0.5)-- (-1,0);
\draw (0,-0.5)-- (-0.5,0);
\draw (-1,-1)-- (-0.5,-0.5);
\draw (-0.5,-1)-- (0,-0.5);
\draw (-1,-0.5)-- (-0.5,-1);
\draw (-0.5,-0.5)-- (0,-1);
\draw (-0.5,-1.5)-- (-1,-1);
\draw (-0.5,-1.5)-- (0,-1);
\draw (-0.5,-1.5)-- (0.5,-1);
\draw (-0.5,0)-- (0,0.5);
\draw (0,0.5)-- (-1,0);
\draw (0,0.5)-- (0.5,0);
\begin{scriptsize}
\fill [color=black] (-1,1) circle (1.5pt);
\fill [color=black] (0,0.5) circle (1.5pt);
\fill [color=black] (-1,-0.5) circle (1.5pt);
\fill [color=black] (-0.5,-0.5) circle (1.5pt);
\fill [color=black] (0,-0.5) circle (1.5pt);
\fill [color=black] (0.5,-0.5) circle (1.5pt);
\fill [color=black] (-1,0) circle (1.5pt);
\fill [color=black] (-0.5,0) circle (1.5pt);
\fill [color=black] (0,0) circle (1.5pt);
\fill [color=black] (0.5,0) circle (1.5pt);
\fill [color=black] (-0.5,-1.5) circle (1.5pt);
\fill [color=black] (-1,-1) circle (1.5pt);
\fill [color=black] (-0.5,-1) circle (1.5pt);
\fill [color=black] (0,-1) circle (1.5pt);
\fill [color=black] (0.5,-1) circle (1.5pt);
\draw[color=black] (-0.5,-1.8) node {$v_2$};
\draw[color=black] (0.25,-0.5) node {$\cdots$};
\draw[color=black] (0.25,0) node {$\cdots$};
\draw[color=black] (0.25,-1) node {$\cdots$};
\end{scriptsize}
\end{tikzpicture}\label{J2}}
\caption{}
\end{figure}
For $k=3r$, we take $G=H_{3r,n}-v_1$ (Figure \ref{h3rn-v1}). 
In this graph, $N_G((u_1,i))\subseteq N_G((u_3,i))$ for 
all $1\leq i\leq n$, so $I(G)\simeq I(G_1)$ where 
$$G_1=G-\bigcup_{1\leq i\leq n}N_G((u_3,i)).$$
Now, in $G_1$, $N_G((u_4,i))\subseteq N_G((u_6,i))$ for all $1\leq i\leq n$, so 
$I(G_1)\simeq I(G_2)$ where
$$G_2=G_1-\bigcup_{1\leq i\leq n}N_G((u_6,i)).$$
We keep doing this until we get a graph $G_{r}\cong K_1+rK_2\times K_n$ where the isolated vertex is
$v_2$. Therefore $I(G)\simeq*$ and 
$I(H_{3r})\simeq\Sigma I(H_{3r,n}-N_{H_{3r,n}}[v_1])\cong\Sigma I(H_{3r-1,n})$.

For $k=3r+1$, we take $G=H_{3r+1,n}-v_1$ and do the same proces as before, this time 
in $G_r$ the vertex $(u_{3r+1},2)$ is isolated, so $I(G)\simeq*$. Therefore 
$$I(H_{3r+1})\simeq\Sigma I(H_{3r+1,n}-N_{H_{3r,n}}[v_1])\cong\Sigma I(H_{3r,n})\cong
\Sigma^2I(H_{3r-1,n})$$

For $k=3r+2$, by Theorem \ref{barmak}
$$I(H_{k,n})\simeq\Sigma\left(st(v_1)\cap SC(v_1)\right),$$
$$st(v_1)\cap st((u_1,2))\cong I(J_1),$$
with $J_1$ obtained from $W_{k-2,n}$ attaching a leaf to $v_1$ (Figure \ref{J1}), and
$$st(v_1)\cap st((u_1,i))\cong I(J_2),$$
with $J_2=W_{k,n}-N((u_2,3))$ (Figure \ref{J2}).

In $J_1-((u_2,2))$, the vertex $(u_1,1)$ is an isolated vertex, therefore
$$I(J_1)\simeq\Sigma I(J_1-N[(u_1,2)])\cong\Sigma I(W_{3r-1,n})\simeq\bigvee_{n-1}\Sigma^3I(H_{3(r-2)+2}).$$
In $J_2-((u_2,3))$, the vertex $(u_1,1)$ is an isolated vertex, therefore
$$I(J_2)\simeq\Sigma I(J_2-N[(u_2,3)])\cong\Sigma I(H_{3(r-1)+2,n})$$
Now the intersection of any of these complexes is contrctible, because the vertex $(u_1,1)$ is an isolated 
vertex in the corresponding subgraph. Thus, by Lemma \ref{lempegcontra},
$$\Sigma\left(st(v_1)\cap SC(v_1)\right)\simeq\left(\bigvee_{n-1}\Sigma^4I(H_{k-6,n})\right)\vee\left(\bigvee_{n-2}\Sigma^2I(H_{k-3,n})\right).$$
\end{proof}

\begin{lem}\label{homohkn}
For $k\geq2$ and $n\geq3$, $I(H_{k,n})$ has the homotopy type of a wedge of spheres of the following dimension:
\begin{itemize}
\item[(a)] $2r$ if $k=3r$.
\item[(b)] $2r+1$ if $k=3r+1$ or $k=3r+2$.
\end{itemize}
Moreover, for small $k$ we can say how many spheres:
$$I(H_{k,n})\simeq\left\lbrace
\begin{array}{cc}
\displaystyle\bigvee_{n-2}\mathbb{S}^1 & \mbox{if } k=2\\
\displaystyle\bigvee_{n-2}\mathbb{S}^2 & \mbox{if } k=3\\
\displaystyle\bigvee_{n-2}\mathbb{S}^3 & \mbox{if } k=4\\
\displaystyle\bigvee_{(n-1)+(n-2)^2}\mathbb{S}^3 & \mbox{if }k=5
\end{array}
\right.$$
\end{lem}
\begin{proof}
For $k=2$, the neighborhood of $(u_1,2)$ contains the neighborhood of $v_1$, so we can erase $(u_1,2)$. In 
this new graph the neighborhood of $(u_2,1)$ contains the neighborhood of $v_2$, so we can 
erase $(u_2,1)$. Now the neighborhood of $(u_1,1)$ contains the neigborhood of $v_2$, and the one 
of $(u_2,1)$ the one of $v_1$, so we can erase $(u_1,1)$ and $(u_2,2)$. This new graph is isomorphic to 
$K_2\times K_{n-1}$, so 
$$I(H_{2,n})\simeq\bigvee_{n-2}\mathbb{S}^1.$$
For $k=3$, $H_{3,k}-N_{H_{3,k}}[v_1]\cong H_{2,n}$ and $I(H_{3,n}-v_1)\simeq*$, therefore
$$I(H_{3,n})\simeq\bigvee_{n-2}\mathbb{S}^2.$$
For $k=4$, $H_{4,n}-N_{H_{4,n}[v_1]}\cong H_{3,n}$ and $I(H_{4,n}-v_1)\simeq*$, therefore
$$I(H_{4,n})\simeq\bigvee_{n-2}\mathbb{S}^3.$$
For $k=5$, we know that 
$$I(H_{5,n})\simeq\Sigma\left(st(v_1)\cap SC(N(v_1))\right)$$
and that 
$$st(v_1)\cap SC(N(v_1))=\bigcup_{i=1}^{n-1}K_i,$$
where, taking $N(v_1)=\{u_1,\dots,u_{n-1}\}$,
$$K_i=st(v_1)\cap st(u_i)=I\left((G-N(v_1))\cap(G-N(u_i))\right).$$
For $i=1$, $(G-N(v_1))\cap(G-N(u_1))$ is isomorphic to $W_{3,n}$ with a leaf adjacent to 
$v_1$, therefore, erasing all the neighbors of $v_1$ but the leaf, we get that
$$K_1\simeq\Sigma I(W_{2,n})\simeq \Sigma I(K_2\times K_n),$$
so 
$$K_1\simeq\bigvee_{n-1}\mathbb{S}^2.$$
For $i\geq2$, $(G-N(v_1))\cap(G-N(u_i))$ is isomorphic to $H_{3,n}$ with a leaf adjacent to 
$v_1$, therefore, erasing all the neighbors of $v_1$ but the leaf, we get that
$$K_i\simeq\Sigma I(H_{2,n})\simeq\bigvee_{n-2}\mathbb{S}^2.$$
In any intersection of these complexes the leaf becomes an isolated vertex, therefore the intersections 
are contractible, so 
$$st(v_1)\cap SC(N(v_1))\simeq\bigvee_{i=1}^{n-1}K_i.$$
Therefore
$$I(H_{5,n})\simeq\bigvee_{(n-1)+(n-2)^2}\mathbb{S}^3$$
Using $H_{2,n}$ and $H_{5,n}$ as the base for the induction and the last lemma we get that 
$I(H_{k,n})$ has the homotopy type of the wedge of spheres of the desired dimension.
\end{proof}

From Lemma \ref{lemhkn} we see that the homotopy type of $I(H_{k,n})$ only depends of the homotopy 
type of the complexes $I(H_{3r+2,n})$, which is, by Lemma \ref{homohkn}, the wedge of some number of $(2r+1)$-spheres. 
If we let $h(r,n)$ denote to the number of spheres in  $I(H_{3r+2,n})$ we have the following recursion relation:
\begin{itemize}
\item[(a)] $h(0,n)=n-2$
\item[(b)] $h(1,n)=n-1+(n-2)^2=1+h(0,n)+(h(0,n))^2$
\item[(c)] $h(r,n)=(n-1)h(r-2,n)+(n-2)h(r-1,n)$ for $r\geq 2$
\end{itemize}

This recursion can be solved by standard techniques, and better still, once the solution is found, it is easy to verify by induction. The solution works out to be
\begin{equation}\label{defh}
    h(r,n) = \frac{(n-1)^{r+2} - (-1)^r}{n}.
\end{equation}

Now from Lemmas \ref{lemhkn} and \ref{homohkn}, we get:
\begin{lem}\label{lemhomohkn}
$$I(H_{k,n})\simeq\left\lbrace
\begin{array}{cc}
\displaystyle\bigvee_{h(r-1,n)}\mathbb{S}^{2r} & \mbox{if } k=3r\\
\displaystyle\bigvee_{h(r-1,n)}\mathbb{S}^{2r+1} & \mbox{if } k=3r+1\\
\displaystyle\bigvee_{h(r,n)}\mathbb{S}^{2r+1} & \mbox{if } k=3r+2
\end{array}
\right.$$
\end{lem}

Now we can determine the homotopy type of $G_{k,n}$.
\begin{lem}\label{lemgkn}
For $k\geq2$ and $n\geq3$
$$I(G_{k,n})\simeq\left\lbrace
\begin{array}{cc}
\displaystyle\bigvee_{n}\mathbb{S}^2 & \mbox{if } k=2\\
\displaystyle\bigvee_{h(r-1,n)}\mathbb{S}^{2r} & \mbox{if } k=3r\\
\displaystyle\bigvee_{h(r-1,n)}\mathbb{S}^{2r+1} & \mbox{if } k=3r+1\\
\displaystyle\bigvee_{h(r-1,n)+(n-1)^{r+1}}\mathbb{S}^{2r+2} & \mbox{if } k=3r+2 \mbox{ and } r\geq1
\end{array}
\right.$$
\end{lem}
\begin{proof}
For $k=2$, $G_{2,n}-N[v_1]\cong K_2+K_{1,n}$, therefore $I(G_{2,n}-N[v_1])\simeq\mathbb{S}^1$. In 
$G_{2,n}-v_1$ the only neighbor of $w_1$ is $w_2$, so 
$$I(G_{2,n}-v_1)\simeq I(G_{2,n}-v_1-v_2)\cong I(K_2+K_2\times K_n)\simeq\bigvee_{n-1}\mathbb{S}^2,$$
and therefore,
$$I(G_{2,n})\simeq\bigvee_{n}\mathbb{S}^2.$$
For $k=3r,3r+1$, $G_{k,n}-N[w_1]$ is isomorphic to $H_{k,n}-v_1$ and as we saw in the proof of 
the Lemma \ref{lemhkn}, $I(H_{k,n}-v_1)\simeq*$. Therefore 
$$I(G_{k,n})\simeq I(G_{k,n}-w_1).$$
In $G_{k,n}-w_1$, the only neighbor of $w_2$ is $v_2$, so we can erase all the neighbors of $v_2$ except 
$w_2$ and we get
$$I(G_{k,n}-w_1)\simeq I(K_2+H_{k-1,n})\cong\Sigma I(H_{k-1,n}).$$
Using Lemma \ref{lemhomohkn}, we get the result.

For $k=3r+2$ with $r\geq1$, in the graph $G_{k,n}-N[v_1]$ the only neighbor of $w_2$ is $v_2$, 
so we can erase all the neighbors of $v_2$ but for $w_2$ and we get that
$$I(G_{k,n}-N[v_1])\simeq I(K_2+H_{k-2,n})\cong\Sigma I(H_{k-2,n}),$$
which, by Lemma \ref{homohkn}, has the homotopy type of a wedge of $(2r+1)$-spheres.
Now, in $G_{k,n}-v_1$ the only neighbor of $w_1$ is $w_2$, so we can remove $v_2$ and obtain
$$I(G_{k,n}-v_1)\simeq I(K_2+P_k\times K_n)\cong\Sigma I(P_k\times K_n),$$
which, by Theorem \ref{homopkkn}, has the homotopy type of a wedge of $(n-1)^{r+1}$ 
$(2r+2)$-spheres, and thus the inclusion $I(G_{k,n}-N[v_1])\longhookrightarrow I(G_{k,n}-v_1)$ is 
null-homotopic. Therefore,
$$I(G_{k,n})\simeq\Sigma^2I(H_{k-2,n})\vee\Sigma I(P_k\times K_n).$$
By Theorem \ref{homopkkn} and Lemma \ref{lemhomohkn} we get the result.
\end{proof}

For $n\geq3$ and $k\geq2$, we define:
\begin{itemize}
\item $\mathring{W}_{k,n}$ as the graph obtained from $W_{k,n}$ by taking the path of length $3$ 
with vertices $w_1,w,w_2$ and edges $\{\{w_1,w\}\{w,w_2\}\}$ and making $v_1$ adjacent to $w_1$ and 
$v_2$ to $w_2$.
\item $\mathring{H}_{k,n}$ as the graph obtained from $H_{k,n}$ by taking two new vertices $w_1$ and $w_2$, 
and making $v_1$ adjacent to $w_1$ and $v_2$ to $w_2$.
\end{itemize}

\begin{lem}\label{homowcirc}
$$I(\mathring{W}_{k,n})\simeq\left\lbrace
\begin{array}{cc}
\displaystyle\bigvee_{(n-1)^r}\mathbb{S}^{2r+1} & \mbox{if } k=3r\\
* & \mbox{if } k=3r+1\\
\displaystyle\bigvee_{(n-1)^{r+1}}\mathbb{S}^{2r+2} & \mbox{if } k=3r+2
\end{array}
\right.$$
\end{lem}
\begin{proof}
When $k=3r$, in $T=\mathring{W}_{3r,n}-N[w_1]$, the neighborhood of the vertex $(u_1,i)$ is contain in the 
neighborhood of the vertex $(u_3,i)$ for all $i$. Then, we can erase the row $u_3$ form $T$ and the 
independence complex of this new graph is homotopy equivalent to $I(G)$. In this new graph the neighborhood of 
$(u_4,i)$ is contained in the one of $(u_6,i)$, so we can erase the row $u_6$ and the homotopy type will 
not change. Continuing with this process until we have erased all the rows $u_{3k}$ for $1\leq k\leq r$, 
we obtain a graph which is isomorphic to $K_2+rK_2\times K_n$, so 
$$I(T)\simeq\Sigma I(rK_2\times K_n)\simeq\bigvee_{(n-1)^r}\mathbb{S}^{2r}.$$
Now, in $\mathring{W}_{3r,n}-w_1$ the only neighbor of $w$ is $w_2$, so we can erase $v_2$. In this new 
graph, the neighborhood of $(u_{3r},i)$ is contain in the one of $(u_{3r-2},i)$, so we can erase the row 
$u_{3r-2}$. Continuing this process as before, we erase the rows $u_{3k-2}$ for all $1\leq k \leq r$. At 
the end of this, the vertex $v_1$ is an isolated vertex, so $I(\mathring{W}_{3r,n}-w_1)\simeq*$ and 
therefore
$$I(\mathring{W}_{3r,n})\simeq\Sigma I(T)\simeq\bigvee_{(n-1)^r}\mathbb{S}^{2r+1}.$$
For $k=3r+1$, as before we take $T=\mathring{W}_{3r+1,n}-N[w_1]$ and erase the rows $u_{3k}$ 
for $1\leq k\leq r$, we get a graph in which the vertex $(u_{3r+1},2)$ is an isolated vertex, then 
$I(T)\simeq*$ and $I(\mathring{W}_{3r+1,n})\simeq I(\mathring{W}_{3r,n}-w_1)$. In 
$\mathring{W}_{3r+1,n}-w_1$, the only neighbor of $w$ is $w_2$, so we can erase $v_2$. In this new 
graph, the neighborhood of $(u_{3r+1},i)$ is contained in the one of $(u_{3r-1},i)$, so we can erase the row 
$u_{3r-1}$. Continuing this process as before, we erase the rows $u_{3k-1}$ for all $1\leq k \leq r$. At 
the end of this, the vertex $(u_1,2)$ is an isolated vertex, so $I(\mathring{W}_{3r+1,n}-w_1)\simeq*$ and 
therefore 
$$I(\mathring{W}_{3r+1,n})\simeq*.$$
For $k=3r+2$, as before we take $T=\mathring{W}_{3r+2,n}-N[w_1]$ and erase the rows $u_{3k}$ 
for $1\leq k\leq r$. In this graph the neighborhood of $(u_{3r+1},2)$ is contain in the one of 
$v_2$, so we can erase $v_2$ and $w_2$ becomes an isolated vertex. Therefore 
$$I(\mathring{W}_{3r+2,n})\simeq I(\mathring{W}_{3r+2,n}-w_1).$$
In $\mathring{W}_{3r+2,n}-w_1$, the only neighbor of $w$ is $w_2$, so we can erase $v_2$. In this new 
graph, the neighborhood of $(u_{3r+2},i)$ is contain in the one of $(u_{3r},i)$, so we can erase the row 
$u_{3r}$. Continuing this process, we erase the rows $u_{3k}$ for all $1\leq k \leq r$. In this graph 
the neighborhood of $v_1$ is equal to the one of $(u_2,2)$, so we can erase $v_1$, therefore
$$I(\mathring{W}_{3r+2,n})\simeq I(K_2\sqcup(r+1)K_2\times K_n)\simeq\bigvee_{(n-1)^{r+1}}\mathbb{S}^{2r+2}.$$
\end{proof}

\begin{lem}\label{homoHcirc}
$$I(\mathring{H}_{k,n})\simeq\left\lbrace
\begin{array}{cc}
\mathbb{S}^2 & \mbox{if } k=2\\
\mathbb{S}^3 & \mbox{if } k=3\\
\Sigma^2I(H_{k-2,n}) & \mbox{for all } k\geq4
\end{array}
\right.$$
\end{lem}
\begin{proof}
Because $N(w_i)=\{v_i\}$, we can erase the vertices $(u_1,i)$ and $(u_k,j)$ for $i>1$ and $j\neq2$. 
Now
\begin{enumerate}
\item If $k=2$, the resulting graph is isomorphic to $3K_2$.
\item If $k\geq4$, the resulting graph is isomorphic to $2K_2\sqcup H_{k-2,n}$.
\item If $k=3$, the only neighbor of $(u_2,1)$ in the resulting graph is $(u_3,2)$, so we can erase 
al the vertices $(u_2,i)$ with $i>2$. This new graph is isomorphic to $4K_2$.
\end{enumerate}
\end{proof}

Before we prove the main result of this section, we need the next lemma. 
\begin{lem}\label{descompcompl}
For $v=(u_1,1)\in V(C_r\times K_n)$, the complex $st(v)\cap SC(v)$ is 
the union of complexes $X_1,\dots,X_{n-1},Y_1,\dots,Y_{n-1}$, where 
\begin{enumerate}
\item For any $i$, $X_i\cong Y_i\cong I(G_{r-4,n})$.
\item For any $i$ and $r\geq7$, $X_i\cap Y_i\cong I(\mathring{W}_{r-5,n})$.
\item For any $i\neq j$ and $r\geq7$, $X_i\cap Y_j\cong I(\mathring{H}_{r-5,n})$.
\item For any $i\neq j$, $X_i\cap X_j\simeq*\simeq Y_i\cap Y_j$.
\item For any $L_1,\dots,L_m$, with $m\geq3$ and 
$L_i\in\{X_1,\dots,X_{n-1},Y_1,\dots,Y_{n-1}\}$, we have
$$\bigcap_{j=1}^mL_j\simeq*$$
\end{enumerate}
\end{lem}
\begin{proof}
By definition 
$$SC(v)=\bigcup_{u\in N(v)}st(u)$$
and in $C_r\times K_n$, $|N(v)|=2(n-1)$, we call 
$$X_i=st(v)\cap st((u_2,i+1))$$
and 
$$Y_i=st(v)\cap st((u_n,i+1))$$
Now, $X_i$ is the independence complex of the induced subgraph given by de set 
$$S_i=V(C_r\times K_n)-\left(N(v)\cup N((u_2,i+1))\right)$$
where, taking $w=(u_2,i+1)$,
$$N(v)\cup N(w)=\{(u_2,j):\:j>1\}\cup\{(u_n,j):\:j>1\}\cup\{(u_1,j):\:j\neq i+1\}\cup\{(u_3,j):\:j\neq i+1\}$$
therefore $(C_r\times K_n)[S_i]\cong G_{r-4,n}$.  

Now, $X_i\cap X_j\cong I((C_r\times K_n)[S_i]\cap (C_r\times K_n)[S_j])$, with $i\neq j$, 
in $(C_r\times K_n)[S_i]\cap (C_r\times K_n)[S_j]$ the vertex 
$(u_2,1)$ is an isolated vertex, therefore $X_i\cap X_j\simeq*$. For $Y_i's$ is analogous, with 
$(u_r,1)$ being the isolated vertex. Now, for the intersection of more than $2$ complexes, one or 
both of these vertices are isolated.

Now, taking $S_i$ as before and 
$$R_j=V(C_r\times K_n)-\left(N(v)\cup N((u_r,j+1))\right)$$
taking $t=(u_r,i+1)$ 
$$N(v)\cup N(t)=\{(u_2,l):\:l>1\}\cup\{(u_n,l):\:l>1\}\cup\{(u_1,l):\:l\neq j+1\}\cup\{(u_{r-1},l):\:l\neq j+1\}$$
Then $X_i\cap Y_j\cong I((C_r\times K_n)[S_i\cap R_j])$ and 
\begin{itemize}
\item If $i=j$, then $(C_r\times K_n)[S_i\cap R_j]\cong\mathring{W}_{r-5,n}$.
\item If $i\neq j$, then $(C_r\times K_n)[S_i\cap R_j]\cong\mathring{H}_{r-5,n}$.
\end{itemize}
\end{proof}

\begin{figure}
\centering
\subfigure[$K_i\cap L_i$]{\begin{tikzpicture}[line cap=round,line join=round,>=triangle 45,x=1cm,y=1cm]
\clip(-1.5,-1.5) rectangle (1,2);
\draw (-0.5,1.5)-- (-1,1);
\draw (-1,1)-- (-0.5,0.5);
\draw (-0.5,0.5)-- (-1,0);
\draw (-0.5,0.5)-- (0,0);
\draw (-0.5,0.5)-- (0.5,0);
\draw (-0.5,-0.5)-- (-1,0);
\draw (-0.5,-0.5)-- (0,0);
\draw (-0.5,-0.5)-- (0.5,0);
\draw (-0.5,-0.5)-- (-1,-1);
\draw (-1,-1)-- (-1.25,0.5);
\draw (-1.25,0.5) -- (-1.25,1);
\draw (-1.25,1)-- (-1,1.25);
\draw (-1,1.25)-- (-0.5,1.5);
\begin{scriptsize}
\fill [color=black] (-1,1) circle (1.5pt);
\fill [color=black] (-0.5,0.5) circle (1.5pt);
\fill [color=black] (-0.5,-0.5) circle (1.5pt);
\fill [color=black] (-1,0) circle (1.5pt);
\fill [color=black] (-0.5,0) circle (1.5pt);
\fill [color=black] (0,0) circle (1.5pt);
\fill [color=black] (0.5,0) circle (1.5pt);
\fill [color=black] (-0.5,1.5) circle (1.5pt);
\fill [color=black] (-1,-1) circle (1.5pt);
\draw[color=black] (0.25,0) node {$\cdots$};
\end{scriptsize}
\end{tikzpicture}\label{kili}}
\subfigure[$K_i\cap L_j$]{\begin{tikzpicture}[line cap=round,line join=round,>=triangle 45,x=1cm,y=1cm]
\clip(-1.2,-1.5) rectangle (1,2);
\draw (-0.5,0.5)-- (-1,0);
\draw (-0.5,0.5)-- (0,0);
\draw (-0.5,0.5)-- (0.5,0);
\draw (0,-0.5)-- (-1,0);
\draw (0,-0.5)-- (-0.5,0);
\draw (0,-0.5)-- (0.5,0);
\draw (-1,1)-- (-0.5,0.5);
\draw (-1,-1)-- (0,-0.5);
\begin{scriptsize}
\fill [color=black] (-1,1) circle (1.5pt);
\fill [color=black] (-0.5,0.5) circle (1.5pt);
\fill [color=black] (0,-0.5) circle (1.5pt);
\fill [color=black] (-1,0) circle (1.5pt);
\fill [color=black] (-0.5,0) circle (1.5pt);
\fill [color=black] (0,0) circle (1.5pt);
\fill [color=black] (0.5,0) circle (1.5pt);
\fill [color=black] (-1,-1) circle (1.5pt);
\draw[color=black] (0.25,0) node {$\cdots$};
\end{scriptsize}
\end{tikzpicture}\label{kilj}}
\caption{}
\end{figure}

\begin{theorem}\label{teomod3}
$$I(C_k\times K_n)\simeq\left\lbrace
\begin{array}{cc}
\displaystyle\bigvee_{2(n-1)}\mathbb{S}^1 & \mbox{if } k=3\\
\displaystyle\bigvee_{(n-1)}\mathbb{S}^1 & \mbox{if } k=4\\
\displaystyle\bigvee_{n}\mathbb{S}^2 & \mbox{if } k=5\\
\displaystyle\bigvee_{(n-1)(3n-2)}\mathbb{S}^3 & \mbox{if } k=6\\
\displaystyle\bigvee_{n(n-1)h(r-3,n)+2(n-1)^r}\mathbb{S}^{2r-1} & \mbox{if } k=3r \mbox{ and } r\geq3
\end{array}
\right.$$
\end{theorem}
\begin{proof}
For $k\leq5$, we have already see it. For $k=6$, we now that 
$$I(C_6\times K_n)\simeq\Sigma st((u_1,1))\cap SC((u_1,1))$$
where 
$$SC((u_1,1))=\left(\bigcup_{(u_i,n-1)\atop i>1}st((u_i,n-1))\right)\cup\left(\bigcup_{(u_i,2)\atop i>1}st((u_i,2))\right)$$
Making the intersecction we get $K_1,\dots,K_{n-1},L_1,\dots,L_{n-1}$ complexes which are isomorphic to 
$I(G_{2,n})$, which has the homotopy type of an wedge of $n$ $2$-dimentional spheres. The intersection of 
any $K_i$ and $L_i$ is ismomorphic to the independence complex of the graph in Figure \ref{kili}, which has an 
isolated vertex, so $K_i\cap L_i\simeq*$. For $i\neq j$, the complex $K_i\cap L_j$ is isomorphic to the 
independence complex of the graph in Figure \ref{kilj}, which is homotopy equivalent to $\mathbb{S}^1$. The intersection of three or more of 
these complexes is always contractible. Therefore, $st((u_1,1))\cap SC((u_1,1))$ has the hompotopy type of the 
wedge of $2(n-1)$ copies of $\displaystyle\bigvee_{n}\mathbb{S}^2$ with $(n-1)(n-2)$ copies of 
$\mathbb{S}^2$. Therefore 
$$I(C_6\times K_n)\simeq\bigvee_{(n-1)(3n-2)}\mathbb{S}^3$$
For $k=3r$ with $r\geq3$, by the Lemma \ref{descompcompl}, 
$st(v)\cap SC(v)$ is the union of complexes $K_1,\dots,K_{n-1},L_1,\dots,L_{n-1}$, where 
\begin{enumerate}
\item For any $i$, $K_i\cong L_i\cong I(G_{k-4,n})$.
\item For any $i$, $K_i\cap L_i\cong I(\mathring{W}_{k-5,n})$.
\item For any $i\neq j$, $K_i\cap L_j\cong I(\mathring{H}_{k-5,n})$.
\item For any $i\neq j$, $K_i\cap K_j\simeq*\simeq L_i\cap L_j$.
\item For any $X_1,\dots,X_l$, with $l\geq3$ and 
$X_i\in\{K_1,\dots,K_{n-1},L_1,\dots,L_{n-1}\}$, we have
$$\bigcap_{j=1}^lX_j\simeq*$$
\end{enumerate}
So we have $2(n-1)$ copies of $I(G_{3(r-3)+2,n})$, which has the homotopy type of the wedge of 
$h(r-3,n)+(n-1)^{r-1}$ copies of $\mathbb{S}^{2r-2}$, $(n-1)(n-2)$ copies of 
$I(\mathring{H}_{3(r-2)+1,n})$ which has the homotopy type of 
$$\Sigma^2I(H_{3(r-3)+2,n})\simeq\bigvee_{h(r-3,n)(n-2)}\mathbb{S}^{2r-3}$$
and $n-1$ copies of $I(\mathring{W}_{3(r-2)+1,n})\simeq*$. By Proposition \ref{homocolimpegadoesferas} and  
taking its suspension we get that 
$$I(C_{3r}\times K_n)\simeq\bigvee_{n(n-1)h(r-3,n)+2(n-1)^r}\mathbb{S}^{2r-1}$$
\end{proof}

\begin{theorem}\label{teonotmod3}
For $r\geq2$: 
\begin{itemize}
\item[(a)] $\pi_1(I(C_{3r+1}\times K_n))\cong0\cong \pi_1(I(C_{3r+2}\times K_n))$.
\item[(b)] $\tilde{H}_{q}(I(C_{3r+1}\times K_n))\cong0$ for all $q\neq 2r,2r-1$.
\item[(c)] $\tilde{H}_{q}(I(C_{3r+2}\times K_n))\cong0$ for all $q\neq 2r+1,2r$.
\item[(d)] $I(C_{3r+1}\times K_n)$ has the homotopy type of a wedge of $2r$-spheres, $2r+1$-spheres and moore spaces of the type
$M(\mathbb{Z}_m,2r)$.
\item[(e)] $I(C_{3r+2}\times K_n)$ has the homotopy type of a wedge of $2r+1$-spheres, $2r+2$-spheres and moore spaces of the type
$M(\mathbb{Z}_m,2r+1)$.
\end{itemize}
\end{theorem}
\begin{proof}
From Lemma \ref{descompcompl} and Theorem \ref{barmak}, for $s=1,2$, 
$I(C_{3r+s}\times K_n)\simeq\Sigma(X\cup Y)$, where 
$$X\cong\bigcup_{i=1}^{n-1}X_i,\; \; Y\cong\bigcup_{i=1}^{n-1}Y_i,\; \;X_i\cong I(G_{3r+s-4,n})\cong Y_i$$
and 
$$\bigcap_{i\in S}X_i\simeq*\simeq\bigcap_{i\in S}Y_i$$
for any $S\subset\{1,\dots,n-1\}$ and $|S|\geq2$. 
By Proposition \ref{homocolimpegadoesferas}, Lemmas \ref{descompcompl} and \ref{lemgkn}, 
$$X\simeq \bigvee_{_{(n-1)h(r-2,n)}}\mathbb{S}^{2r-3+s}\simeq Y$$
By the Seifert–van Kampen Theorem (see \cite{rotmantop} Theorem 7.40), $\pi_1(X\cup Y)\cong0$.

Now, for $C_{3r+1}\times K_n$, by Lemma \ref{descompcompl}, 
$$X\cap Y=\bigcup_{1\leq i,j\leq n-1}X_i\cap Y_j$$
where
$$X_i\cap Y_i\cong I(\mathring{W}_{3(r-2)+2,n})$$ 
$$X_i\cap Y_j\cong I(\mathring{H}_{3(r-2)+2,n}) \mbox{ for } i\neq j$$ 
and $(X_i\cap Y_j)\cap(X_r\cap Y_s)\simeq*$. By Proposition \ref{homocolimpegadoesferas}
$$X\cap Y\simeq\left(\bigvee_{(n-2)(n-1)}I\left(\mathring{H}_{3(r-2)+2,n}\right)\right)\vee\left(\bigvee_{n-1}I\left(\mathring{W}_{3(r-2)+2,n}\right)\right)$$
By  Lemmas \ref{lemhomohkn}, \ref{homowcirc} and \ref{homoHcirc}a,
$$X\cap Y\simeq\bigvee_{_{(n-1)^r+(n-2)(n-1)h(r-3,n)}}\mathbb{S}^{2r-2}.$$
Then, by the Mayer-Vietoris sequence, taking $K=X\cup Y$,
\begin{equation*}
\xymatrix{
0 \ar@{->}[r] &\tilde{H}_{2r-1}(K) \ar@{->}[r] & \mathbb{Z}^{l} \ar@{->}[r] &\mathbb{Z}^{d}\oplus\mathbb{Z}^{d} \ar@{->}[r] & \tilde{H}_{2r-2}(K) \ar@{->}[r] & 0
}
\end{equation*}
where $l=(n-1)^r+(n-2)(n-1)h(r-3,n)$ and $d=(n-1)h(r-2,n)$. Therefore 
$\tilde{H}_{q}(K)\cong0$ for $q\neq 2r-1,2r-2$ and taking the suspension we get the result.
For $C_{3r+2}\times K_n$ is analogous.

Parts (d) and (e) follow form the previous parts (see example 4C.2 \citep{hatcher}).
\end{proof}

Proposition \ref{propCnK2} tell us that for any $k$ $I(C_k\times K_2)$ has the homotopy type of a wedge of spheres of the same 
dimension and Theorem 
\ref{teomod3} tell us that for $k=3r$ and any $n$ this is also true, so one can ask what happen for the other $k$'s. For 
$k\not\equiv 0 \pmod{3}$, the last Theorem tell us that the complex may have nontrivial homology only in two consecutive dimensions; and 
by calculations done with Sage we know that for $C_7\times K_3,C_7\times K_4,C_7\times K_5,C_8\times K_3,C_{10}\times K_3,C_{10}\times K_3$, 
their independence complexes have non-trivial free homology groups in these two dimensions. From all this we can ask the following question:

\begin{que}
    Are the homology groups of $I(C_k\times K_n)$ always torsion-free?
\end{que}

An afirmative answer would tell us that $I(C_k\times K_n)$ always has the homotopy type of a wedge of spheres.

\section{Independence Complex of $K_2\times K_n\times K_m$}
In this section we will calculate the homotopy type of $I(K_2\times K_n\times K_m)$. We take the following polynomial.
$$f(l,n,m)=\frac{(l-1)(n-1)(m-1)(lnm-4)}{4}$$

\begin{prop}
$$I(K_2\times K_n\times K_m)\simeq\bigvee_{f(2,n,m)}\mathbb{S}^3$$
\end{prop}
\begin{proof}
We take $G=K_2\times K_n\times K_m$, then: 
\begin{itemize}
\item If either $n$ or $m$ is equal to $1$, then $f(2,n,m)=0$.
\item If $n=2$ then 
$$I(G)=I(K_2\times K_m)*I(K_2\times K_m)\simeq\bigvee_{(m-1)^2}\mathbb{S}^3$$
and the formula holds. The same is true for $m=2$.
\item If $n=3$, then $f(2,3,m)$ is the formula for $C_6\times K_m$. The same for $m=3$.
\end{itemize} 
Assume that $n,m\geq3$. Because $K_2$ is one of 
the factors in the product, $G$ has no $K_3$ and by Theorem \ref{barmak}
$$I(G)\simeq\Sigma(st((1,1,1))\cap SC((1,1,1)))$$
As before, $\displaystyle st((1,1,1))\cap SC((1,1,1))=\bigcup_{v\in N_G((1,1,1))}(st((1,1,1))\cap st(v))$

\begin{figure}
\centering
\begin{tikzpicture}[line cap=round,line join=round,>=triangle 45,x=2cm,y=2cm]
\clip(-2.5,-1.7) rectangle (2.5,1.7);
\draw (-0.75,1)-- (0.6,0.2);
\draw (0.75,1)-- (-0.6,0.2);
\draw (-0.75,1)-- (0.75,-0.2);
\draw (0.75,1)-- (-0.75,-0.2);
\draw (-0.75,1)-- (0.6,-1);
\draw (0.75,1)-- (-0.6,-1);
\draw (0.6,0.2)-- (-0.75,-0.2);
\draw (-0.6,0.2)-- (0.75,-0.2);
\draw (0.6,0.2)-- (-0.6,-1);
\draw (-0.6,0.2)-- (0.6,-1);
\draw (-0.75,-0.2)-- (0.6,-1);
\draw (0.75,-0.2)-- (-0.6,-1);
\draw (-1.75,0.6)-- (-1.25,1);
\draw (1.75,0.6)-- (1.25,1);
\draw (-1.75,-0.6)-- (-1.25,-0.2);
\draw (1.75,-0.6)-- (1.25,-0.2);
\draw (-1.75,0.6)-- (-1.4,0.2);
\draw (1.75,0.6)-- (1.4,0.2);
\draw (-1.75,-0.6)-- (-1.4,-1);
\draw (1.75,-0.6)-- (1.4,-1);
\draw (-1.75,0)-- (-1.25,1);
\draw (-1.75,0)-- (-1.25,-0.2);
\draw (-1.75,0)-- (-1.4,0.2);
\draw (-1.75,0)-- (-1.4,-1);
\draw (1.75,0)-- (1.25,1);
\draw (1.75,0)-- (1.25,-0.2);
\draw (1.75,0)-- (1.4,0.2);
\draw (1.75,0)-- (1.4,-1);
\draw (-2,0.67) .. controls (-1,2) and (1,2) .. (2,0.67);
\draw (-2,-0.67) .. controls (-1,-2) and (1,-2) .. (2,-0.67);
\draw (-1.75,0)-- (1.75,0);
\begin{scriptsize}
\draw[color=black] (-1,1) node {$(1,2,3)$};
\draw[color=black] (1,1) node {$(2,1,3)$};
\draw[color=black] (-1,0.2) node {$(1,2,m)$};
\draw[color=black] (1,0.2) node {$(2,1,m)$};
\draw[color=black] (-1,-0.2) node {$(1,3,2)$};
\draw[color=black] (1,-0.2) node {$(2,3,1)$};
\draw[color=black] (-1,-1) node {$(1,n,2)$};
\draw[color=black] (1,-1) node {$(2,n,1)$};
\draw[color=black] (1,0.5) node {$\vdots$};
\draw[color=black] (1,0.7) node {$\vdots$};
\draw[color=black] (-1,0.5) node {$\vdots$};
\draw[color=black] (-1,0.7) node {$\vdots$};
\draw[color=black] (1,-0.5) node {$\vdots$};
\draw[color=black] (1,-0.7) node {$\vdots$};
\draw[color=black] (-1,-0.5) node {$\vdots$};
\draw[color=black] (-1,-0.7) node {$\vdots$};
\draw[color=black] (-2,0.6) node {$(2,1,2)$};
\draw[color=black] (-2,-0.6) node {$(2,2,1)$};
\draw[color=black] (2,0.6) node {$(1,2,1)$};
\draw[color=black] (2,-0.6) node {$(1,1,2)$};
\draw[color=black] (-2,0) node {$(2,1,1)$};
\draw[color=black] (2,0) node {$(1,2,2)$};
\end{scriptsize}
\end{tikzpicture}
\caption{$G[S]=Q_{n,m}$}\label{st000st111}
\end{figure}
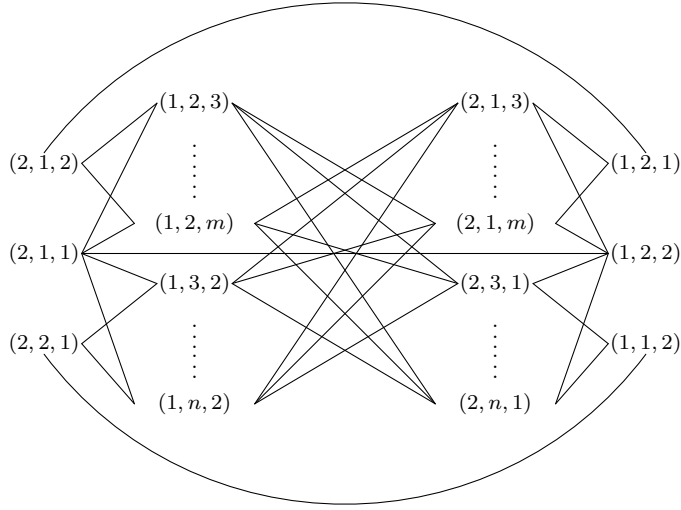

We will first see that each of the complexes of this union is homotopy equivalent to the wedge of 
$n+m-3$ spheres. For $(2,2,2)$, $st((1,1,1))\cap st((2,2,2))=I(G[S])$ with 
$S=V(G)-N_G((1,1,1))\cup N_G((2,2,2))$. Now 
$S=\sigma\cup\tau$, with
$$\sigma=\{(1,2,1),(1,2,2),\dots,(1,2,m),(1,1,2),(1,2,2),\dots,(1,n,2)\}$$
and
$$\tau=\{(2,1,1),(2,1,2),\dots,(2,1,m),(2,1,1),(2,2,1),\dots,(2,n,1)\}$$
therefore, $G[S]$ is the graph in Figure \ref{st000st111}, which we will 
call $Q_{n,m}$. Beacuse $n,m\geq3$, $m+n\geq6$. 
If $m+n=6$, then $m,n=3$ and if we remove the vertex $(2,1,1)$ and its neighbors, we get the disjoint 
union of two copies of $P_3$, therefore $I(Q_{3,3}-N_{Q_{3,3}}((2,1,1))\simeq\mathbb{S}^1$. Now, 
$Q_{3,3}-(2,1,1)$ is isomorphic to $C_8$ plus a vertex $v$ adjacent to two vertices in $C_8$ 
which are at distance $4$; if we remove this vertex and its neighbors we get two disjoint copies of 
$P_3$, therefore, by Proposition \ref{cofseq}, $I(Q_{3,3}-(2,1,1))\simeq\mathbb{S}^2\vee\mathbb{S}^2$. 
Then, again by Proposition \ref{cofseq}, $\displaystyle I(Q_{3,3})\simeq\bigvee_{3}\mathbb{S}^2$. 
Assume that for all $6\leq n+m\leq k$, $\displaystyle I(Q_{n,m})\simeq\bigvee_{n+m-3}\mathbb{S}^2$ and 
take $Q_{n,m}$ such that $n+m=k+1$, without loss of generality assume that $m\geq4$. Now, in 
$F=Q_{n,m}-N_{Q_{n,m}}[(1,2,3)]$ the only neighbor of $(1,2,1)$ is $(2,1,3)$, therefore 
$I(F)\simeq I(F-R)$ with 
$$R=\{(1,2,2),(1,2,4),\dots,(1,2,m),(1,3,1),\dots,(1,n,2)\}$$
and $F-R\cong M_2$, and thus $I(F)\simeq\mathbb{S}^1$. Now, in $T=Q_{n,m}-(1,2,3)$ 
$N_T((2,1,4))\subseteq N_T((2,1,3))$, by Lemma \ref{vecindad}, 
$I(T)\simeq I(T-(2,1,3))$. Because $T-(2,1,3)\cong Q_{n,m-1}$, by the inductive hypothesis, 
$$I(T)\simeq\bigvee_{n+m-4}\mathbb{S}^2,$$
and by Proposition \ref{cofseq}, 
$$I(Q_{n,m})\simeq\bigvee_{n+m-3}\mathbb{S}^2.$$

Now, 
$$(st((1,1,1))\cap st(v))\cap(st((1,1,1))\cap st(u))=st((1,1,1))\cap st(u)\cap st(v)=I(G[A]),$$
with $A=V(G)-N_G((1,1,1))\cup N_G(u)\cap N_G(v)$. There are two possibilities
\begin{itemize}
\item $u$ and $v$ have two coordinates equal. Assume that $u=(2,a,b)$ and $v=(2,a,c)$, with 
$b,c>1$ and $b\neq c$. Take $(2,a,1),(x,y,z)\in A$. If $x=1$, then $y=a$ because $b\neq c$. 
Therefore $(2,a,1)(x,y,z)\notin E(G)$ for all $(x,y,z)\in A$ and $I(G[A])\simeq*$.
\item $u$ and $v$ have only on coordinate equal. Assume $u=(2,a,b)$ and $v=(2,c,d)$, with 
$a\neq c$, $b\neq d$ and $a,b,c,d>1$. Then 
$$A=\{(1,a,d),(1,c,b),(2,1,1),(2,1,2),\dots,(2,1,m),(2,2,1),\dots,(2,n,1)\}$$
In $G[A]$, the only neighborhood of $(2,1,d)$ is $(1,c,b)$ and only one of $(2,c,1)$ is 
$(1,a,d)$, so we can erase all other vertices without changing the homotopy type, and therefore 
$I(G[A])\simeq I(M_2)\cong\mathbb{S}^1$.
\end{itemize}
Therefore, the inclusion of the intersection of two complexes of the union 
$st((1,1,1))\cap SC((1,1,1))$ is null-homotopic. Now, the intersection of three complexes is equal 
to $I(G[D])$ with $D=V(G)-N_g((1,1,1))\cup N_G(u_1)\cap N_G(u_2)\cap N_G(u_3)$. There are two 
possibilities
\begin{itemize}
\item If three vertices have only the first coordinate equal, $(2,a,b),(2,c,d),(2,e,f)$, then
the only vertices with the first coordinate equal to $1$ that are not neighbors of $(2,a,b)$ or 
$(2,c,d)$ are $(1,a,d)$ and $(1,c,b)$, which are neighbors of $(2,e,f)$, therefore 
$$D=\{2\}\times V(K_n)\times V(K_m)$$
and
$$I(G[D])\simeq*.$$
\item If two vertices have two coordinates equal, the intersection is a cone as with only two vertices.
\end{itemize}
Then the union $st((1,1,1))\cap SC((1,1,1))$ achieve the hypothesis of Proposition \ref{homocolimpegadoesferas}. Now 
$(1,1,1)$ has $(n-1)(m-1)$ neighbors and for each neighbor there are another $(n-2)(m-2)$ neighbors 
differing in two coordinates, these pairs are counted twice, therefore 
$st((1,1,1))\cap SC((1,1,1))$ is homotopy equivalent to the wedge of 
$$(n-1)(m-1)(n+m-3)+\frac{(n-1)(m-1)(n-2)(m-2)}{2}=\frac{(n-1)(m-1)(2mn-4)}{4}$$ 
spheres, and taking the suspension we arrive at the result.
\end{proof}
We finish with the following conjecture.
\begin{conj}
$$I(K_n\times K_m \times K_l)\simeq\bigvee_{f(n,m,l)}\mathbb{S}^3$$
\end{conj}

\bibliographystyle{acm}
\bibliography{tipohomotopias}

\begin{thebibliography}{10}

\bibitem{adamsplit}
{\sc Adamaszek, M.}
\newblock Splittings of independence complexes and the powers of cycles.
\newblock {\em Journal of Combinatorial Theory, Series A 112}, 5 (2012),
  1031--1047.

\bibitem{introhomo}
{\sc Arkowitz, M.}
\newblock {\em Introduction to Homotopy Theory}.
\newblock Universitext. Springer, 2011.

\bibitem{barmak}
{\sc Barmak, J.~A.}
\newblock Star clusters in independence complexes of graphs.
\newblock {\em Advances in Mathematics 241\/} (2013), 33--57.

\bibitem{bjornertopmeth}
{\sc Björner, A.}
\newblock {\em Handbook of Combinatorics}.
\newblock North-Holland, 1995, ch.~Topological methods, pp.~1819--1872.

\bibitem{engstrom09}
{\sc Engström, A.}
\newblock Complexes of directed trees and independence complexes.
\newblock {\em Discrete Mathematics 309\/} (2009), 3299--3309.

\bibitem{homotopygoyal}
{\sc Goyal, S., Shukla, S., and Singh, A.}
\newblock Homotopy type of independence complexes of certain families of
  graphs.
\newblock {\em Contributions to Discrete Mathematics 16}, 3 (2021), 74--92.

\bibitem{handbookprod}
{\sc Hammack, R., Imrich, W., and Klav\v{z}ar, S.}
\newblock {\em Handbook of Product Graphs}.
\newblock Discrete Mathematics and Its Applications. CRC Press, 2016.

\bibitem{hatcher}
{\sc Hatcher, A.}
\newblock {\em Algebraic Topology}.
\newblock Cambridge, 2002.

\bibitem{kozlovdire}
{\sc Kozlov, D.}
\newblock Complexes of directed trees.
\newblock {\em Journal of Combinatorial Theory, Series A 88\/} (1999),
  112--122.

\bibitem{cubicalhomotopy}
{\sc Munson, B.~A., and Voli\'c, I.}
\newblock {\em Cubical Homotopy Theory}, vol.~25 of {\em New mathematical
  monographs}.
\newblock Cambridge, 2015.

\bibitem{rotmantop}
{\sc Rotman, J.~J.}
\newblock {\em An Introduction to Algebraic Topology}.
\newblock Springer, 1998.

\bibitem{MR2022345}
{\sc Wachs, M.~L.}
\newblock Topology of matching, chessboard, and general bounded degree graph
  complexes.
\newblock {\em Algebra Universalis 49}, 4 (2003), 345--385.

\end{thebibliography}

\vspace{1cm}

Omar Antol\'in Camarena \hspace{3cm}Andr\'es Carnero Bravo

Instituto de Matem\'aticas,  \hspace{2.75cm}Instituto de Matem\'aticas,

UNAM, Mexico City, Mexico    \hspace{2.25cm}UNAM, Mexico City, Mexico

\hspace{7.15cm}\textit{E-mail address:} \href{mailto:acarnerobravo@gmail.com}{acarnerobravo@gmail.com}
\end{document}